\documentclass[11pt]{article}
\usepackage{amsfonts, amsmath, amssymb, latexsym, eucal, amscd,cite} 
\usepackage[all]{xy}
\newtheorem{theorem}{Theorem}[section]
\newtheorem{lemma}[theorem]{Lemma}
\newtheorem{proposition}[theorem]{Proposition}

\newtheorem{exAux}[theorem]{Example}
\newenvironment{example}{\begin{exAux} \rm}{\end{exAux}}
\newtheorem{Def}[theorem]{Definition}
\newenvironment{definition}{\begin{Def} \rm}{\end{Def}}
\newtheorem{Note}[theorem]{Note}
\newenvironment{note}{\begin{Note} \rm}{\end{Note}}
\newtheorem{Problem}[theorem]{Problem}
\newenvironment{problem}{\begin{Problem} \rm}{\end{Problem}}
\newtheorem{Rem}[theorem]{Remark}

\newtheorem{Not}[theorem]{Notation}

\newtheorem{Conj}[theorem]{Conjecture}

\newtheorem{Ass}[theorem]{Assumption}

\newenvironment{proof}{\medskip\noindent{\bf Proof.\ }}{\qed\medskip}
\newenvironment{proofof}[1]{\medskip\noindent{\bf Proof  of {#1}.\ 
}}{\qed\medskip}
\newcommand{\qed}{\hfill\mbox{$\Box$\qquad\qquad}}

\newcommand{\F}{\mathbb{F}}
\newcommand{\Mat}{\text{\rm Mat}}

\newcommand{\vphi}{\varphi}
\renewcommand{\th}{\theta}

\newcommand{\ve}{\varepsilon}

\newcommand{\yb}{\overline{y}}

\newif\ifDRAFT
\DRAFTfalse

\begin{document}

\title{Leonard pairs having zero-diagonal TD-TD form}

\author{Kazumasa Nomura}
\date{}

\maketitle

\bigskip

{
\begin{quote}
\begin{center}
{\bf Abstract}
\end{center}
Fix an algebraically closed field $\F$ and an integer $n \geq 1$.
Let $\Mat_n(\F)$ denote the $\F$-algebra consisting of the $n \times n$ matrices
that have all entries in $\F$.
We consider a pair of diagonalizable  matrices in $\Mat_{n}(\F)$,
each acting in an irreducible tridiagonal fashion on an 
eigenbasis for the other one.
Such a pair is called a Leonard pair in $\Mat_{n}(\F)$.
In the present paper, we find
all Leonard pairs $A,A^*$ in $\Mat_{n}(\F)$ such that each of $A$ and
$A^*$ is irreducible tridiagonal with all diagonal entries $0$.
This solves a problem given by Paul Terwilliger.
\end{quote}
}

\section{Introduction}

Throughout the paper $\F$ denotes an algebraically closed field.
Fix an integer $d \geq 0$ and a vector space $V$ over $\F$ with dimension $d+1$.
Let $\F^{d+1}$ denote the $\F$-vector space consisting of the column vectors
of length $d+1$, and $\Mat_{d+1}(\F)$ denote the $\F$-algebra consisting of
the $(d+1) \times (d+1)$ matrices that have all entries in $\F$.
We index rows and columns by $0,1,\ldots,d$.
The algebra $\Mat_{d+1}(\F)$ acts on $\F^{d+1}$ by left multiplication.

We begin by recalling the notion of a Leonard pair.
We use the following terms.
A square matrix is said to be {\em tridiagonal} whenever each nonzero
entry lies on either the diagonal, the subdiagonal, or the superdiagonal.
A tridiagonal matrix is said to be {\em irreducible} whenever
each entry on the subdiagonal is nonzero and each entry on the superdiagonal is nonzero.

\begin{definition}  {\rm (See \cite[Definition 1.1]{T:Leonard}.)}    \label{def:LP}   \samepage
\ifDRAFT {\rm def:LP}. \fi
By a {\em Leonard pair on $V$} we mean an ordered pair of linear transformations
$A : V \to V$ and $A^*: V \to V$ that satisfy (i) and (ii) below:
\begin{itemize}
\item[\rm (i)]
There exists a basis for $V$ with respect to which the matrix representing
$A$ is irreducible tridiagonal and the matrix representing $A^*$ is diagonal.
\item[\rm (ii)]
There exists a basis for $V$ with respect to which the matrix representing
$A^*$ is irreducible tridiagonal and the matrix representing $A$ is diagonal.
\end{itemize}
We say $A,A^*$ has {\em diameter} $d$.
By a {\em Leonard pair in $\Mat_{d+1}(\F)$} we mean an ordered pair of matrices
$A,A^*$ in $\Mat_{d+1}(\F)$ that acts on $\F^{d+1}$ as a Leonard pair. 
\end{definition}

\begin{note}    \samepage
According to a common notational convention, $A^*$ denotes the
conjugate transpose of $A$.
We are not using this convention.
In a Leonard pair $A,A^*$ the matrices $A$ and $A^*$ are arbitrary subject to
the conditions (i) and (ii) above.
\end{note}

We refer the reader to
\cite{NT:affine, NT:balanced, NT:span, T:Leonard, T:array, T:tworelations, T:LBUB, T:survey, TV, V:AW}
for background on Leonard pairs.
Paul Terwilliger gave the following problems.

\begin{problem}  {\rm (See \cite[Problem 36.14]{T:survey}.)}   \label{prob:LBTD}    \samepage
\ifDRAFT {\rm prob:LBTD}. \fi
Find all Leonard pairs $A,A^*$ in $\Mat_{d+1}(\F)$ that satisfy the
following conditions:
(i) $A$ is lower bidiagonal with all subdiagonal entries $1$;
(ii) $A^*$ is irreducible tridiagonal.
\end{problem}

\begin{problem} {\rm (See \cite[Problem 36.16]{T:survey}.)}   \label{prob}
\ifDRAFT {\rm prob}. \fi
Find all Leonard pairs $A,A^*$ in $\Mat_{d+1}(\F)$ 
such that each of $A,A^*$ is irreducible tridiagonal with all diagonal
entries $0$.
\end{problem}

In \cite{N:LBTD} we gave a partial solution of Problem \ref{prob:LBTD}.
In the present paper we solve Problem \ref{prob}.
To state our main results, we first recall the notion of an isomorphism of Leonard pairs.
Let $A,A^*$ be a Leonard pair on $V$ and let $B,B^*$ be a Leonard pair
on a vector space $V'$ with dimension $d+1$.
By an {\em isomorphism of Leonard pairs} from $A,A^*$ to $B,B^*$
we mean a linear bijection $\sigma : V \to V'$ such that
both $\sigma A = B \sigma$ and $\sigma A^* = B^* \sigma$.
We say two Leonard pairs $A,A^*$ and $B,B^*$ are {\em isomorphic}
whenever there exists an isomorphism of Leonard pairs
form $A,A^*$ to $B,B^*$.
We use the following term:

\begin{definition}   \label{def:matTDTD}   \samepage
\ifDRAFT {\rm def:matLBTD}. \fi
A matrix $A \in \Mat_{d+1}(\F)$ is said to be {\em zero-diagonal TD} whenever
$A$ is irreducible tridiagonal with all diagonal entries $0$.
A pair of matrices $A,A^*$ in $\Mat_{d+1}(\F)$ is said to be {\em zero-diagonal TD-TD} 
whenever each of $A,A^*$ is zero-diagonal TD.
\end{definition}

\begin{note}   \samepage
The following hold for nonzero $\xi, \xi^* \in \F$.
\begin{itemize}
\item[\rm (i)]
Let $A,A^*$ be a Leonard pair on $V$.
Then $\xi A$, $\xi^* A^*$ is a Leonard pair on $V$.
\item[\rm (ii)]
Let $A,A^*$ be a zero-diagonal TD-TD pair in $\Mat_{d+1}(\F)$.
Then $\xi B$, $\xi^* B$ is a zero-diagonal TD-TD pair in $\Mat_{d+1}(\F)$.
\end{itemize}
\end{note}

\begin{definition}    \samepage
Let $A,A^*$ be a Leonard pair on $V$.
By the {\em opposite} of $A,A^*$ we mean the Leonard pair $-A,-A^*$.
\end{definition}

We are now ready to state our first main result.

\begin{theorem}    \label{thm:main}    \samepage
\ifDRAFT {\rm thm:main}. \fi
Let $A,A^*$ be a Leonard pair on $V$.
Then the following {\rm (i)} and {\rm (ii)} are equivalent:
\begin{itemize}
\item[\rm (i)]
There exists a basis for $V$ with respect to which the matrices representing $A,A^*$
form a zero-diagonal TD-TD pair in $\Mat_{d+1}(\F)$.
\item[\rm (ii)]
$A,A^*$ is isomorphic to its opposite.
\end{itemize}
\end{theorem}

\begin{note}    \label{note:(i)-(ii)}   \samepage
\ifDRAFT {\rm note:(i)-(ii)}. \fi
In Theorem \ref{thm:main}, the implication (i)$\Rightarrow$(ii) is immediate from
the following observation.
Consider the diagonal matrix $D \in \Mat_{d+1}(\F)$ that has $(i,i)$-entry $(-1)^i$
for $0 \leq i \leq d$.
Let $A \in \Mat_{d+1}(\F)$ be a zero-diagonal TD matrix.
Then $D^{-1} A D = - A$.
\end{note}

Theorem \ref{thm:main} is related to a class of Leonard pairs, called totally bipartite.
It is known that a totally bipartite Leonard pair is isomorphic to its opposite.
(see \cite[Chapter 2, Lemma 38]{T:lec}).
See \cite{Brown, HWG, Mik, T:lec} for more information concerning
totally bipartite Leonard pairs.

Below we describe the the parameter array 
of a Leonard pair that is isomorphic to its opposite
(see Definition \ref{def:parray} for the definition of a parameter array).

\begin{proposition}    \label{prop:-A-As}    \samepage
\ifDRAFT {\rm prop:-A-As}. \fi
Let $A,A^*$ be a Leonard pair on $V$ with parameter array
\[
   (\{\th_i\}_{i=0}^d, \{\th^*_i\}_{i=0}^d, \{\vphi_i\}_{i=1}^d, \{\phi_i\}_{i=1}^d).
\]
Then the following {\rm (i)} and {\rm (ii)} are equivalent:
\begin{itemize}
\item[\rm (i)]
$A,A^*$ is isomorphic to its opposite.
\item[\rm (ii)]
The parameter array satisfies
\begin{align*}
  & \th_i + \th_{d-i} = 0,   & & \th^*_i  + \th^*_{d-i}=0   && (0 \leq i \leq d),
\\
  & \vphi_i = \vphi_{d-i+1},  &&  \phi_i = \phi_{d-i+1}     &&  (1 \leq i \leq d).
\end{align*}
\end{itemize}
\end{proposition}

We handle the case $d \leq 2$ in Section \ref{sec:dleq2}.
For the rest of this section, assume $d \geq 3$.
In this case, the fundamental parameter $\beta$ is well-defined
(see Definition \ref{def:beta} for the definition).
In \cite{T:array} Terwilliger gave a classification of Leonard pairs.
By that classification,  Leonard pairs are classified into 
thirteen types.
For a Leonard pair that is isomorphic to its opposite,
the type is as follows
(see Definition \ref{def:types} for the definition
of these types).

\begin{proposition}    \label{prop:types}    \samepage
\ifDRAFT {\rm prop:types}. \fi
Let $A,A^*$ be a Leonard pair on $V$ that is isomorphic to its opposite.
Let $\beta$ be the fundamental parameter of $A,A^*$.
\begin{itemize}
\item[\rm (i)]
Assume $\beta = 2$. Then $A,A^*$ has Krawtchouk type.
\item[\rm (ii)]
Assume $\beta = -2$. Then $A,A^*$ has Bannai-Ito type with even diameter.
\item[\rm (iii)]
Assume $\beta \neq 2$ and $\beta \neq -2$.
Then $A,A^*$ has $q$-Racah type.
\end{itemize}
\end{proposition}

In Section \ref{sec:list} we display five families of zero-diagonal TD-TD Leonard pairs
in $\Mat_{d+1}(\F)$. See Propositions \ref{prop:type2ex}--\ref{prop:type1even}.
Among these five families, the family in Proposition \ref{prop:type1compact}
is the most general one.
This family comes from the ``compact basis''
given by Ito-Rosengren-Terwilliger  (see \cite[Section 17]{IRT}).
The compact basis is obtained from an evaluation module for the $q$-tetrahedron algebra.
See \cite{Funk, IRT, IT:qtetra, Miki}  about the $q$-tetrahedron algebra.
The families in Propositions \ref{prop:type2ex}, \ref{prop:type3ex}, \ref{prop:type1LT}
are related to ``Leonard triples''.
See \cite{Brown, Cur, GWH, HWG, HZG, Huang, KZ} about Leonard triples.
The family in Proposition \ref{prop:type1even} is somewhat mysterious,
and the author has no conceptual explanation for this family.

\begin{proposition}   \label{prop:(ii)->(i)}     \samepage
\ifDRAFT {\rm prop:(ii)->(i)}. \fi
Let $A,A^*$ be a Leonard pair on $V$ that is isomorphic to its opposite.
Let $\beta$ be the fundamental parameter of $A,A^*$.
Then after replacing $A,A^*$ with their nonzero scalar multiples if necessary,
the following hold.
\begin{itemize}
\item[\rm (i)]
Assume $\beta = 2$.
Then $A,A^*$ is represented by a zero-diagonal TD-TD pair in $\Mat_{d+1}(\F)$
that belongs to the family in Proposition \ref{prop:type2ex}.
\item[\rm (ii)]
Assume $\beta = -2$.
Then  $A,A^*$ is represented by a zero-diagonal TD-TD pair in $\Mat_{d+1}(\F)$ 
that belongs to the family in Proposition \ref{prop:type3ex}.
\item[\rm (iii)]
Assume $\beta \neq 2$ and $\beta \neq -2$.
Then $A,A^*$ is represented by a zero-diagonal TD-TD pair in $\Mat_{d+1}(\F)$ 
that belongs to the family in Proposition \ref{prop:type1compact}.
\end{itemize}
\end{proposition}

Theorem \ref{thm:main}(ii)$\Rightarrow$(i) immediately follows from Proposition \ref{prop:(ii)->(i)}.
To state our second main result, we make some observations and definitions.

\begin{note}   \label{note:equivalent}   \samepage
\ifDRAFT {\rm note:equivalent}. \fi
Let $A,A^*$ be a zero-diagonal TD-TD pair in $\Mat_{d+1}(\F)$,
and let $D \in \Mat_{d+1}(\F)$ be an invertible diagonal matrix.
Then $D^{-1} A D$, $D^{-1} A^* D$ is a zero-diagonal TD-TD pair in $\Mat_{d+1}(\F)$.
Moreover, if $A,A^*$ is a Leonard pair, then $D^{-1} A D$, $D^{-1} A^* D$ is a Leonard
pair that is isomorphic to $A,A^*$
\end{note}

\begin{definition}    \label{def:equivalent}   \samepage
\ifDRAFT {\rm def:equivalent}. \fi
Let $A,A^*$ and $B,B^*$ be zero-diagonal TD-TD pairs in $\Mat_{d+1}(\F)$.
We say $A,A^*$ and $B,B^*$ are {\em equivalent}
whenever there exists an invertible diagonal matrix $D \in \Mat{d+1}(\F)$
such that $B=D^{-1} A D$ and $B^* = D^{-1} A^* D$.
\end{definition}

\begin{note}   \label{note:subdiagonal1}    \samepage
\ifDRAFT {\rm note:subdiagonal1}. \fi
Let $A \in \Mat_{d+1}(\F)$ be a zero-diagonal TD-TD matrix that has subdiagonal entries
$\{x_i\}_{i=1}^d$.
Let $D \in \Mat_{d+1}(\F)$ be the diagonal matrix that has $(i,i)$-entry $x_1x_2\cdots x_i$
for $0 \leq i \leq d$.
Then $D^{-1} A D$ is a zero-diagonal TD-TD matrix that has all subdiagonal entries $1$.
\end{note}

\begin{note}   \label{note:reverse}   \samepage
\ifDRAFT {\rm note:reverse}. \fi
Let $A \in \Mat_{d+1}(\F)$ be a zero-diagonal TD matrix with
subdiagonal entries $\{x_i\}_{i=1}^d$ and superdiagonal entries $\{y_i\}_{i=1}^d$.
Then the anti-diagonal transpose of $A$ has subdiagonal entries $\{x_{d-i+1}\}_{i=1}^d$
and superdiagonal entries $\{y_{d-i+1}\}_{i=1}^d$.
Observe that the anti-diagonal transpose of $A$ is $Z^{-1} A^\textsf{T} Z$,
where $Z \in \Mat_{d+1}(\F)$ has $(i,j)$-entry $\delta_{i,d-j}$ for $0 \leq i,j \leq d$,
and $A^\textsf{T}$ denotes the transpose of $A$.
Let $A,A^*$ be a zero-diagonal TD-TD Leonard pair in $\Mat_{d+1}(\F)$.
By \cite[Theorem 2.2]{T:survey} the anti-diagonal transpose of $A$ and $A^*$
form a Leonard pair that is isomorphic to $A,A^*$. We call this Leonard pair
the {\em anti-diagonal transpose} of $A,A^*$.
\end{note}

We are now ready to state our second main result:

\begin{theorem}    \label{thm:main2}    \samepage
\ifDRAFT {\rm thm:main2}. \fi
Let $A,A^*$ be a zero-diagonal TD-TD Leonard pair in $\Mat_{d+1}(\F)$
with fundamental parameter $\beta$.
\begin{itemize}
\item[\rm (i)]
Assume $\beta=2$.
Then, after replacing $A,A^*$ with their nonzero scalar multiples if necessary,
$A,A^*$ is equivalent to a zero-diagonal TD-TD pair that belongs to the family 
in Proposition \ref{prop:type2ex}.
\item[\rm (ii)]
Assume $\beta = -2$.
Then, after replacing $A,A^*$ with their nonzero scalar multiples if necessary,
$A,A^*$ is equivalent to a zero-diagonal TD-TD pair that belongs to
the family in Proposition \ref{prop:type3ex}.
\item[\rm (iii)]
Assume $\beta \neq 2$ and $\beta \neq -2$. 
Then, after replacing $A,A^*$ with their nonzero scalar multiples if necessary,
$A,A^*$ or its anti-diagonal transpose is equivalent to a zero-diagonal TD-TD pair
that belongs to one of the families in Propositions 
\ref{prop:type1compact}--\ref{prop:type1even}.
\end{itemize}
\end{theorem}

The paper is organized as follows.
In Section \ref{sec:LS} we recall some materials concerning Leonard pairs.
In Section \ref{sec:properties} we  prove Proposition \ref{prop:-A-As}.
In Section \ref{sec:dleq2} we handle the case $d \leq 2$.
In Sections \ref{sec:types}--\ref{sec:type1}
we assume $d \geq 3$.
In Section \ref{sec:types} we recall some formulas that represent the parameter array
in closed form.
In Section \ref{sec:parray} we display formulas for the parameter array
of a Leonard pair that is isomorphic to its opposite.
Using these formulas we prove Proposition \ref{prop:types}.
In Section \ref{sec:list} we display five families of zero-diagonal TD-TD
Leonard pairs in $\Mat_{d+1}(\F)$.
In Section \ref{sec:AWrel} we recall the Askey-Wilson relations for a Leonard pair.
In Section \ref{sec:char} we display a formula for the characteristic polynomial
of a zero-diagonal TD matrix in $\Mat_{d+1}(\F)$.
In Sections \ref{sec:type2ex}--\ref{sec:type1even} we prove Propositions
\ref{prop:type2ex}--\ref{prop:type1even}.
In Section \ref{sec:proofmain} we prove Proposition \ref{prop:(ii)->(i)}.
In Section \ref{sec:evalAWrel} we evaluate the Askey-Wilson relations 
for a zero-diagonal TD-TD Leonard pair in $\Mat_{d+1}(\F)$, and obtain
some relations between the entries of the matrices.
In Section \ref{sec:equat} we obtain some equations for later use.
In Sections \ref{sec:type2}--\ref{sec:type1} we prove Theorem \ref{thm:main2}.

\section{Leonard systems}
\label{sec:LS}

When working with a Leonard pair, it is convenient to consider a closely
related object called a {\em Leonard system}. 
To prepare for our definition
of a Leonard system, we recall a few concepts from linear algebra.

Let $A: V \to V$ be a linear transformation.
We say $A$ is {\em multiplicity-free}
whenever it has $d+1$ mutually distinct eigenvalues in $\F$. 
Assume $A$ is multiplicity-free, and let $\{\th_i\}_{i=0}^d$ be
the eigenvalues of $A$.
For $0 \leq i \leq d$ define
\[
    E_i = \prod_{ \scriptsize \begin{matrix} 0 \leq \ell \leq d \\ \ell \neq i \end{matrix} }
             \frac{A-\theta_\ell I}{\theta_i - \theta_\ell}.
\]
Here $I$ denotes the identity.
Observe
(i) $AE_i = \theta_i E_i$ $(0 \leq i \leq d)$;
(ii) $E_i E_j = \delta_{i,j} E_i$ $(0 \leq i,j \leq d)$;
(iii) $I = \sum_{i=0}^{d} E_i$;
(iv) $A = \sum_{i=0}^{d} \theta_i E_i$.
Also observe
$V = \sum_{i=0}^d E_i V$ (direct sum),
and $E_i$ acts on $V$ as the projection onto $E_i V$.
We call $E_i$ the {\em primitive idempotent} of $A$ associated with
$\theta_i$. 
We now define a Leonard system.

\medskip

\begin{definition}  \cite{T:Leonard}     \label{def:LS}
\ifDRAFT {\rm def:LS}. \fi
By a {\em Leonard system} on $V$ we mean a sequence
\[
   (A, \{E_i\}_{i=0}^d, A^*, \{E^*_i\}_{i=0}^d)
\]
that satisfies (i)--(v) below.
\begin{itemize}
\item[(i)] 
Each of $A$, $A^*$ is a multiplicity-free linear transformation from $V$ to $V$.
\item[(ii)] 
$\{E_i\}_{i=0}^d$ is an ordering of the primitive idempotents of $A$.
\item[(iii)]
$\{E^*_i\}_{i=0}^d$ is an ordering of the primitive idempotents of $A^*$.
\item[(iv)] 
$E_i A^* E_j = 0$ and $E^*_i A E^*_j = 0$ if $|i-j|>1$ for  $0 \leq i,j \leq d$.
\item[(v)]
$E_i A^* E_j \neq 0$ and $E^*_i A E^*_j \neq 0$ if $|i-j|=1$ for $0 \leq i,j \leq d$.
\end{itemize}
\end{definition}

Leonard systems are related to Leonard pairs as follows.
Let $(A, \{E_i\}_{i=0}^d, A^*, \{E^*_i\}_{i=0}^d)$ be a Leonard system on $V$.
Then $A,A^*$ is a Leonard pair on $V$.
Conversely, let $A,A^*$ be a Leonard pair on $V$. 
Then each of $A,A^*$ is multiplicity-free (see \cite[Lemma 1.3]{T:Leonard}).
Moreover there exists an ordering $\{E_i\}_{i=0}^d$ of the
primitive idempotents of $A$, and 
there exists an ordering $\{E^*_i\}_{i=0}^d$ of the
primitive idempotents of $A^*$, such that
$(A, \{E_i\}_{i=0}^d, A^*, \{E^*_i\}_{i=0}^d)$
is a Leonard system on $V$.
We say the Leonard pair $A,A^*$ and the Leonard system
$(A, \{E_i\}_{i=0}^d, A^*, \{E^*_i\}_{i=0}^d)$ are {\em associated}.

\begin{definition}    \label{def:eigenseq}    \samepage
\ifDRAFT {\rm def:eigenseq}. \fi
Let $\Phi=(A, \{E_i\}_{i=0}^d, A^*, \{E^*_i\}_{i=0}^d)$ be a Leonard system on $V$.
For $0 \leq i \leq d$ let $\th_i$ (resp.\ $\th^*_i$) be the eigenvalue of 
$A$ (resp.\ $A^*$) associated with $E_i$ (resp.\ $E^*_i$).
We call $\{\th_i\}_{i=0}^d$ (resp.\ $\{\th^*_i\}_{i=0}^d$) the {\em eigenvalue sequence}
(resp.\ {\em dual eigenvalue sequence}) of $\Phi$.
\end{definition}

We recall the notion of an isomorphism of Leonard systems.
Consider a Leonard system $\Phi = (A, \{E_i\}_{i=0}^d, A^*, \{E^*_i\}_{i=0}^d)$  on $V$
and a Leonard system $\Phi' = (A', \{E'_i\}_{i=0}^d, A^{*\prime}, \{E^{* \prime}_i\}_{i=0}^d)$
on a vector space $V'$ with dimension $d+1$.
By an {\em isomorphism of Leonard systems from $\Phi$ to $\Phi'$} we mean
a linear bijection $\sigma : V \to V'$
such that $\sigma A = A' \sigma$, $\sigma A^* = A^{*\prime} \sigma$,
and $\sigma E_i = E'_i \sigma$, $\sigma E^*_i = E^{*\prime}_i \sigma$ for $0 \leq i \leq d$.
Leonard systems $\Phi$ and $\Phi'$ are said to be {\em isomorphic}
whenever there exists an isomorphism of Leonard systems from
$\Phi$ to $\Phi'$.

Let $A,A^*$ be a Leonard pair on $V$ and let
$\Phi = (A, \{E_i\}_{i=0}^d, A^*, \{E^*_i\}_{i=0}^d)$
be a Leonard system associated with $A,A^*$.
Then $A,A^*$ is associated with the following Leonard systems, and no further Leonard systems:
\begin{align*}
\Phi^{\downarrow} &:= (A, \{E_i\}_{i=0}^d, A^*, \{E^*_{d-i}\}_{i=0}^d),
\\
\Phi^{\Downarrow} &:= (A, \{E_{d-i}\}_{i=0}^d, A^*, \{E^*_{i}\}_{i=0}^d),
\\
\Phi^{\downarrow\Downarrow} &:= (A, \{E_{d-i}\}_{i=0}^d, A^*, \{E^*_{d-i}\}_{i=0}^d).
\end{align*}

Let $\Phi = (A, \{E_i\}_{i=0}^d, A^*, \{E^*_i\}_{i=0}^d)$ be a Leonard system on $V$
with eigenvalue sequence $\{\th_i\}_{i=0}^d$ and dual eigenvalue sequence $\{\th^*_i\}_{i=0}^d$.

\begin{definition}  {\rm (See \cite[Section 5.1]{T:survey}.)}   \label{def:splitbasis}   \samepage
\ifDRAFT {\rm def:splitbasis}. \fi
Pick a nonzero $v \in E^*_0 V$.
For $0 \leq i \leq d$ define
\[ 
  u_i = (A-\th_{i-1} I) \cdots (A-\th_1 I)(A- \th_0 I) v.  
\]
Then $\{u_i\}_{i=0}^d$ is a basis for $V$.
We call $\{u_i\}_{i=0}^d$ a  {\em $\Phi$-split basis} for $V$.
\end{definition}

\begin{lemma} {\rm (See \cite[Theorem 3.2]{T:Leonard}.) }     \label{lem:split}   \samepage
\ifDRAFT {\rm lem:split}. \fi
Let $\{u_i\}_{i=0}^d$ be a $\Phi$-split basis for $V$.
Then the matrices representing $A,A^*$ with respect to $\{u_i\}_{i=0}^d$ are
\begin{align}              \label{eq:splitAAs} 
A &: \;\; 
 \begin{pmatrix}
  \th_0  & & & & & \text{\bf 0} \\
  1 & \th_1  \\
   & 1 & \th_2 \\
   & & \cdot & \cdot \\
   & & & \cdot & \cdot \\
  \text{\bf 0}  & & & & 1 & \th_{d}
 \end{pmatrix},
& 
A^*  &: \;\;
 \begin{pmatrix}
  \th^*_0 & \vphi_1 & & & & \text{\bf 0} \\
     & \th^*_1 & \vphi_2  \\
   &  & \th^*_2 & \cdot  \\
   & &    & \cdot & \cdot \\
   & & &    & \cdot & \vphi_d \\
  \text{\bf 0} & & & &  & \th^*_d
 \end{pmatrix},
\end{align}
for some scalars $\{\vphi_i\}_{i=1}^d$.
The sequence $\{\vphi_i\}_{i=0}^d$ is uniquely determined.
Moreover $\vphi_i \neq 0$ for $1 \leq i \leq d$.
\end{lemma}

\begin{definition}   \label{def:splitseq}   \samepage
\ifDRAFT {\rm def:splitseq}. \fi
With reference to Lemma \ref{lem:split},
we call $\{\vphi_i\}_{i=1}^d$ the {\em first split sequence} of $\Phi$.
By the {\em second split sequence} of $\Phi$ we mean the first split sequence
of $\Phi^\Downarrow$.
\end{definition}

\begin{definition}   {\rm (See \cite[Section 2]{T:array}.)}   \label{def:parray}
\ifDRAFT {\rm def:parray}. \fi
Let $\Phi = (A, \{E_i\}_{i=0}^d, A^*, \{E^*_i\}_{i=0}^d)$  be a Leonard system.
By the {\em parameter array} of $\Phi$ we mean the sequence
\begin{equation}            \label{eq:parray}
   (\{\th_i\}_{i=0}^d, \{\th^*_i\}_{i=0}^d, \{\vphi_i\}_{i=1}^d, \{\phi_i\}_{i=1}^d),
\end{equation}
where $\{\th_i\}_{i=0}^d$ (resp.\ $\{\th^*_i\}_{i=0}^d$) is the eigenvalue sequence
(resp.\ dual eigenvalue sequence) of $\Phi$,
and $\{\vphi_i\}_{i=1}^d$ (resp.\ $\{\phi_i\}_{i=1}^d$) is the first split sequence
(resp.\ second split sequence) of $\Phi$.
\end{definition}

\begin{definition}  \label{def:parrayLP}   \samepage
\ifDRAFT {\rm def:parrayLP}. \fi
Let $A,A^*$ be a Leonard pair on $V$.
By a {\em parameter array} of $A,A^*$ we mean the parameter array of
a Leonard system associated with $A,A^*$.
\end{definition}

\begin{lemma}  {\rm (See  \cite[Theorem 4.6]{NT:formula}.)}  \label{lem:vphiphi}    \samepage
\ifDRAFT {\rm lem:vphiphi}. \fi
Let  $(A, \{E_i\}_{i=0}^d, A^*, \{E^*_i\}_{i=0}^d)$ 
 be a Leonard system with parameter array
$ (\{\th_i\}_{i=0}^d, \{\th^*_i\}_{i=0}^d, \{\vphi_i\}_{i=1}^d, \{\phi_i\}_{i=1}^d)$.
Then for $1 \leq i \leq d$
\begin{align}
 \vphi_i &= (\th^*_0 - \th^*_i)
              \frac{ \text{\rm tr} \big( E^*_0 \prod_{\ell=0}^{i-1} (A-\th_\ell I) \big) }
                     { \text{\rm tr} \big( E^*_0 \prod_{\ell=0}^{i-2} (A-\th_\ell I) \big) },  \label{eq:vphiformula}
\\
 \phi_i &= (\th^*_0 - \th^*_i)
              \frac{ \text{\rm tr} \big( E^*_0 \prod_{\ell=0}^{i-1} (A-\th_{d-\ell} I) \big) }
                     { \text{\rm tr} \big( E^*_0 \prod_{\ell=0}^{i-2} (A-\th_{d-\ell} I) \big) }.  \label{eq:phiformula}
\end{align}
\end{lemma}

The following two results are fundamental in the theory of Leonard pairs.

\begin{lemma}   {\rm (See  \cite[Theorem 1.9]{T:Leonard}.) } 
\label{lem:unique}    \samepage
\ifDRAFT {\rm lem:unique}. \fi
A Leonard system is determined up to isomorphism by its
parameter array.
\end{lemma}

\begin{lemma}  {\rm (See \cite[Theorem 1.9]{T:Leonard}.) } \label{lem:classify}
\ifDRAFT {\rm lem:classify}. \fi
Consider a sequence \eqref{eq:parray} consisting of scalars taken from $\F$.
Then there exists a Leonard system $\Phi$ on $V$ with parameter array 
\eqref{eq:parray} if and only if {\rm (i)--(v)} hold below:
\begin{itemize}
\item[\rm (i)]  
$\th_i \neq \th_j$, $\;\; \th^*_i \neq \th^*_j\;\;$
    $\;\;(0 \leq i < j \leq d)$.
\item[\rm (ii)]
$\vphi_i \neq 0$, $\;\; \phi_i \neq 0\;\;$ $\;\; (1 \leq i \leq d)$.
\item[\rm (iii)]
$ \displaystyle
 \vphi_i = \phi_1 \sum_{\ell=0}^{i-1} 
                            \frac{\th_\ell - \th_{d-\ell}}
                                   {\th_0 - \th_d}
             + (\th^*_i - \th^*_0)(\th_{i-1} - \th_d)  \qquad (1 \leq i \leq d).
$
\item[\rm (iv)]
$ \displaystyle
 \phi_i = \vphi_1 \sum_{\ell=0}^{i-1} 
                            \frac{\th_\ell - \th_{d-\ell}}
                                   {\th_0 - \th_d}
              + (\th^*_i - \th^*_0)(\th_{d-i+1} - \th_0)  \qquad (1 \leq i \leq d).
$
\item[\rm (v)]
The expressions
\begin{equation}    \label{eq:indep}
   \frac{\th_{i-2} - \th_{i+1}}
          {\th_{i-1}-\th_{i}},
 \qquad\qquad
   \frac{\th^*_{i-2} - \th^*_{i+1}}
          {\th^*_{i-1} - \th^*_{i}}
\end{equation}
are equal and independent of $i$ for $2 \leq i \leq d-1$.
\end{itemize}
\end{lemma}

\begin{definition}   \label{def:parrayF}   \samepage
\ifDRAFT {\rm def:parrayF}. \fi
By a {\em parameter array over $\F$} we mean a sequence
\eqref{eq:parray}
consisting of scalars taken from $\F$ that satisfy 
conditions {\rm (i)--(v)} in Lemma \ref{lem:classify}.
\end{definition}

\begin{definition}    \label{def:beta}    \samepage
\ifDRAFT {\rm def:beta}. \fi
Assume $d \geq 3$, and let $\Phi$ be a Leonard system on $V$ with 
parameter array \eqref{eq:parray}.
Let $\beta$ be one less than the common value of \eqref{eq:indep}.
We call $\beta$ the {\em fundamental parameter} of $\Phi$.
Let $A,A^*$ be a Leonard pair on $V$.
By the {\em fundamental parameter} of $A,A^*$ we mean the
fundamental parameter of an associated Leonard system.
\end{definition}

\begin{lemma}   {\rm (See \cite[Theorem 1.11]{T:Leonard}.) }  \label{lem:D4}  \samepage
\ifDRAFT {\rm lem:D4}. \fi
Let $\Phi = (A, \{E_i\}_{i=0}^d, A^*, \{E^*_i\}_{i=0}^d)$ 
 be a Leonard system with parameter array \eqref{eq:parray}.
Then the parameter array of $\Phi^\downarrow$, $\Phi^\Downarrow$, $\Phi^{\downarrow\Downarrow}$
are as follows:
\[
 \begin{array}{c|c}
  \text{\rm Leonard system} & \text{\rm Parameter array} 
 \\ \hline  \rule{0mm}{6mm}
   \Phi & 
   (\{\th_i\}_{i=0}^d, \{\th^*_i\}_{i=0}^d, \{\vphi_i\}_{i=1}^d, \{\phi_i\}_{i=1}^d)
 \\   \rule{0mm}{6mm}
  \Phi^\downarrow &
    (\{\th_i\}_{i=0}^d, \{\th^*_{d-i}\}_{i=0}^d, \{\phi_{d-i+1}\}_{i=1}^d, \{\vphi_{d-i+1}\}_{i=1}^d)
 \\ \rule{0mm}{6mm}
    \Phi^\Downarrow & 
   (\{\th_{d-i}\}_{i=0}^d, \{\th^*_i\}_{i=0}^d, \{\phi_i\}_{i=1}^d, \{\vphi_i\}_{i=1}^d)
 \\ \rule{0mm}{6mm}
  \Phi^{\downarrow\Downarrow} &
    (\{\th_{d-i}\}_{i=0}^d, \{\th^*_{d-i}\}_{i=0}^d, \{\vphi_{d-i+1}\}_{i=1}^d, \{\phi_{d-i+1}\}_{i=1}^d)
 \end{array}
\]
\end{lemma}

We recall the scalars $\{a_i\}_{i=0}^d$ and $\{a^*_i\}_{i=0}^d$.

\begin{definition} {\rm (See \cite[Definition 2.3]{T:survey}.) }   \label{def:ai}    \samepage
\ifDRAFT {\rm def:ai}. \fi
Let  $\Phi = (A, \{E_i\}_{i=0}^d, A^*, \{E^*_i\}_{i=0}^d)$ be a Leonard system on $V$.
Define scalars $\{a_i\}_{i=0}^d$ and $\{a^*_i\}_{i=0}^d$ by
\begin{align*}
  a_i &= \text{\rm tr} (E^*_i A)   &&  (0 \leq i \leq d),
\\
 a^*_i &= \text{\rm tr} (E_i A^*)  &&  (0 \leq i \leq d).
\end{align*}
\end{definition}

\begin{lemma}   {\rm (See \cite[Lemma 2.8]{T:survey}.) }  \label{lem:pricipal}
\ifDRAFT {\rm lem:pricipal}. \fi
With reference to Definition \ref{def:ai}, 
for $0 \leq i \leq d$ pick a nonzero $v_i \in E^*_i V$.
Then $\{v_i\}_{i=0}^d$ be a basis for $V$.
With respect to this basis, the matrix representing $A$ is irreducible tridiagonal 
with diagonal entries $\{a_i\}_{i=0}^d$, and the matrix representing $A^*$
is diagonal with diagonal entries $\{\th^*_i\}_{i=0}^d$, where $\{\th^*_i\}_{i=0}^d$
is the dual eigenvalue sequence of $\Phi$.
\end{lemma}

\begin{lemma}   {\rm (See \cite[Theorem 5.7]{T:survey}.) }  \label{lem:ai}   \samepage
\ifDRAFT {\rm lem:ai}. \fi
With reference to Definition \ref{def:ai}, let \eqref{eq:parray}
be the parameter array of $\Phi$.
Then
\begin{align}
a_i &= \th_i + \frac{\vphi_i}{\th^*_i - \th^*_{i-1}} + \frac{\vphi_{i+1}}{\th^*_i-\th^*_{i+1}}
    && (0 \leq i \leq d),                                                                 \label{eq:ai}
\\
a^*_i &= \th^*_i + \frac{\vphi_i}{\th_i - \th_{i-1}} + \frac{\vphi_{i+1}}{\th_i-\th_{i+1}}
    && (0 \leq i \leq d),        
\end{align}
where $\vphi_0 = 0$, $\vphi_{d+1}=0$, and $\th_{-1}$, $\th_{d+1}$,
$\th^*_{-1}$, $\th^*_{d+1}$ denote indeterminates.
\end{lemma}

We recall a scalar multiple of a Leonard system.

\begin{lemma}   {\rm (See \cite[Lemma 6.1]{NT:affine}.)}  \label{lem:affine}  \samepage
\ifDRAFT {\rm lem:affine}. \fi
Let  $ (A, \{E_i\}_{i=0}^d, A^*, \{E^*_i\}_{i=0}^d)$ 
 be a Leonard system with parameter array \eqref{eq:parray}.
Let. $\xi$,  $\xi^*$ be nonzero scalars in $\F$.
Then
\[
  (\xi A, \, \{E_i\}_{i=0}^d, \, \xi^* A^*, \, \{E^*_i\}_{i=0}^d)
\]
is a Leonard system with parameter array
\[
 (\{\xi \th_i \}_{i=0}^d, \{\xi^* \th^*_i\}_{i=0}^d, 
  \{ \xi\xi^* \vphi_i\}_{i=1}^d, \{\xi\xi^* \phi_i\}_{i=1}^d).
\]
\end{lemma}

In Definition \ref{def:LS} the condition (v) can be slightly weaken as follows.
Let $\text{End}(V)$ denote the $\F$-algebra consisting of the linear
transformations from $V$ to $V$.

\begin{lemma}    \label{lem:irred}   \samepage
\ifDRAFT {\rm lem:irred}. \fi
Consider a sequence $\Phi = (A, \{E_i\}_{i=0}^d, A^*, \{E^*_i\}_{i=0}^d)$
that satisfies conditions {\rm (i)--(iv)} in Definition \ref{def:LS}.
Then the following {\rm (i)--(iii)} are equivalent:
\begin{itemize}
\item[\rm (i)]
$E_i A^* E_j \neq 0\;$ if $\;|i-j|=1$  $\;(0 \leq i,j \leq d)$.
\item[\rm (ii)]
$E^*_i A E^*_j \neq 0\;$ if $\;|i-j|=1$ $\;(0 \leq i,j \leq d)$.
\item[\rm (iii)]
$A$ and $A^*$ together generate $\text{\rm End}(V)$.
\end{itemize}
Suppose {\rm (i)--(iii)} hold above. Then $\Phi$ is a Leonard system.
\end{lemma}

\begin{proof}
The last assertion is clear. We show (ii)$\Leftrightarrow$(iii).
The proof of (i)$\Leftrightarrow$(iii) is similar.

(ii)$\Rightarrow$(iii):
For $0 \leq i \leq d$ pick a nonzero $v_i \in E^*_i V$,
and note that $\{v_i\}_{i=0}^d$ is a basis for $V$.
We identify each linear transformation with the matrix in $\Mat_{d+1}(\F)$
that represents it with respect to $\{v_i\}_{i=0}^d$.
Adopting this point of view, 
$A$ is irreducible tridiagonal and  $A^*$ is diagonal.
Moreover, $E^*_0$ has $(0,0)$-entry $1$ and all other entries $0$.
Using these comments, one finds that 
\begin{align*}
  (A^r E^*_0 A^s) &=
   \begin{cases}
      0  &  \text{ if $\;i>r\;$ or $\;j>s$},
   \\
     \neq 0  & \text{ if $\;i=r\;$ and $\;j=s$}
   \end{cases}
  &&  (0 \leq i,j \leq d).
\end{align*}
Therefore the elements $\{A^r E^*_0 A^s \,|\; 0 \leq r,s \leq d\}$
are linearly independent, and so form a basis for $\Mat_{d+1}(\F)$.
Observe that $E^*_0$ is a polynomial in $A^*$ by the definition.
By these comments $A,A^*$ together generate $\Mat_{d+1}(\F)$.

(iii)$\Rightarrow$(ii): 
By way of contradiction, assume $E^*_{r} A E^*_{r-1} = 0$ or $E^*_{r-1} A E^*_{r} = 0$
for some $r$ $(1 \leq r \leq d)$.
First assume $E^*_{r}A E^*_{r-1} = 0$.
Then $E^*_k A E^*_\ell = 0$ for $0 \leq \ell < r  \leq k \leq d$
by condition (iv) in Definition \ref{def:LS}.
Set $W = \sum_{\ell=0}^{r-1} E^*_\ell V$, and note that $0 \neq W \neq V$.
We claim $W$ is invariant under each of $A$, $A^*$.
Clearly $W$ is invariant under $A^*$.
Using the above comment, we argue
$A W =  A \sum_{\ell=0}^{r-1} E^*_\ell V 
          = I A  \sum_{\ell=0}^{r-1} E^*_\ell V 
   \subseteq \sum_{k=0}^d  \sum_{\ell=0}^{r-1} E^*_k A E^*_\ell V
    = \sum_{k=0}^{r-1} \sum_{\ell=0}^{r-1} E^*_k A E^*_\ell V
   \subseteq \sum_{k=0}^{r-1} E^*_k V = W$.
Therefore $W$ is invariant under $A$.
We have shown the claim.
By the assumption, $A$ and $A^*$ generate $\text{End}(V)$,
so $W$ is invariant under $\text{End}(V)$.
This forces $W=V$, a contradiction.
Next assume $E^*_{r-1} A E^*_r = 0$.
By considering the subspace $W' = \sum_{\ell=r}^d E^*_\ell V$,
we get a contradiction in a similar way as above.
\end{proof}

\section{Some properties of a Leonard pair that is isomorphic to its opposite}
\label{sec:properties}

In this section we study about the parameter array of a Leonard pair
that is isomorphic to its opposite.
We then prove Proposition \ref{prop:-A-As}.
The case $d=0$ is obvious, so we assume $d \geq 1$.
Let
\[
 \Phi = (A, \{E_i\}_{i=0}^d, A^*, \{E^*_i\}_{i=0}^d)
\]
be a Leonard system on $V$ with parameter array
\begin{equation}    \label{eq:parray3}
 (\{\th_i\}_{i=0}^d, \{\th^*_i\}_{i=0}^d, \{\vphi_i\}_{i=1}^d, \{\phi_i\}_{i=1}^d).
\end{equation}

\begin{lemma}   \label{lem:Phid}  \samepage
\ifDRAFT {\rm lem:Phid}. \fi
Define
\begin{equation}
   \Phi' = (-A, \{E_i\}_{i=0}^d, -A^*, \{E^*_i\}_{i=0}^d).        \label{eq:Phid}
\end{equation}
Then $\Phi'$ is a Leonard system with parameter array
\begin{equation}                                                              \label{eq:parrayd}
   (\{- \th_i\}_{i=0}^d, \{- \th^*_i\}_{i=0}^d, \{\vphi_i\}_{i=1}^d, \{\phi_i\}_{i=1}^d).
\end{equation}
\end{lemma}

\begin{proof}
Follows from Lemma \ref{lem:affine}.
\end{proof}

\begin{lemma}    \label{lem:PhidD}    \samepage
\ifDRAFT {\rm lem:PhidD}. \fi
Assume $A,A^*$ is isomorphic to its opposite.
Then 
\begin{align}
   \th_i + \th_{d-i} &= 0    && (0 \leq i \leq d),       \label{eq:th}
\\
  \th^*_i + \th^*_{d-i} &= 0  &&  (0 \leq i \leq d).   \label{eq:ths}
\end{align}
Moreover, $\Phi^{\downarrow \Downarrow}$ is isomorphic to $\Phi'$,
where $\Phi'$ is from \eqref{eq:Phid}.
\end{lemma}

\begin{proof}
Observe that $\Phi'$ is isomorphic to one of
$\Phi$, $\Phi^\downarrow$, $\Phi^\Downarrow$, $\Phi^{\downarrow\Downarrow}$,
since $-A,-A^*$ is isomorphic to $A,A^*$.
By this and Lemma \ref{lem:D4},
$\{-\th_i\}_{i=0}^d$ coincides with $\{\th_i\}_{i=0}^d$ or $\{\th_{d-i}\}_{i=0}^d$.
If $\{-\th_i\}_{i=0}^d$ coincides with $\{\th_i\}_{i=0}^d$, then $\th_i=0$ for $0 \leq i \leq d$,
 contradicting Lemma \ref{lem:classify}(i).
So $\{-\th_i\}_{i=0}^d$ coincides with $\{\th_{d-i}\}_{i=0}^d$, and \eqref{eq:th} follows.
Similarly \eqref{eq:ths} holds.
Therefore $\Phi'$ is isomorphic to $\Phi^{\downarrow \Downarrow}$.
\end{proof}

\begin{lemma}    \label{lem:PhidD2}    \samepage
\ifDRAFT {\rm lem:PhidD2}. \fi
Assume $A,A^*$ is isomorphic to its opposite.
Then 
\begin{align*}
   \vphi_i &= \vphi_{d-i+1}      && (1 \leq i \leq d),
\\
  \phi_i &= \phi_{d-i+1}         && (1 \leq i \leq d).
\end{align*}
\end{lemma}

\begin{proof}
By Lemma \ref{lem:PhidD} $\Phi'$ and  $\Phi^{\downarrow\Downarrow}$ are isomorphic,
so they have the same parameter array.
By Lemma \ref{lem:D4} the parameter array of $\Phi^{\downarrow\Downarrow}$
is
\[
    (\{\th_{d-i}\}_{i=0}^d, \{\th^*_{d-i}\}_{i=0}^d, \{\vphi_{d-i+1}\}_{i=1}^d, \{\phi_{d-i+1}\}_{i=1}^d).
\]
Now compare this with \eqref{eq:parrayd} to get the results.
\end{proof}

\begin{lemma}    \label{lem:char}   \samepage
\ifDRAFT {\rm lem:char}. \fi
Assume $A,A^*$ is isomorphic to its opposite.
Then $\text{\rm Char}(\F) \neq 2$.
\end{lemma}

\begin{proof}
By Lemma \ref{lem:PhidD} $\th_0 = - \th_d$.
If $\text{Char}(\F)=2$, then $\th_0 = \th_d$, contradicting Lemma \ref{lem:classify}(i).
\end{proof}

\begin{lemma}    \label{lem:thinonzero}    \samepage
\ifDRAFT {\rm lem:thinonzero}. \fi
Assume $A,A^*$ is isomorphic to its opposite.
Then for $0 \leq i \leq d$
\begin{align*}
  \th_i &=
    \begin{cases}
       \neq 0 & \text{ if $i \neq d/2$},
    \\
        0      & \text{ if $i=d/2$},
    \end{cases}        
  &
  \th^*_i &=
    \begin{cases}
       \neq 0 & \text{ if $i \neq d/2$},
    \\
        0      & \text{ if $i=d/2$}.
    \end{cases}        
\end{align*}
\end{lemma}

\begin{proof}
Follows from Lemma \ref{lem:classify}(i) and Lemmas \ref{lem:PhidD}, \ref{lem:char}.
\end{proof}

\begin{proofof}{Proposition \ref{prop:-A-As}}
(i)$\Rightarrow$(ii):
Follows from Lemmas \ref{lem:PhidD} and \ref{lem:PhidD2}.

(ii)$\Rightarrow$(i).
Let $\Phi'$ be from \eqref{eq:Phid}.
We show that $\Phi^{\downarrow\Downarrow}$ and $\Phi'$ has the same parameter array.
By Lemma \ref{lem:D4} the parameter array of $\Phi^{\downarrow\Downarrow}$ is
\[
     (\{\th_{d-i}\}_{i=0}^d, \{\th^*_{d-i}\}_{i=0}^d, \{\vphi_{d-i+1}\}_{i=1}^d, \{\phi_{d-i+1}\}_{i=1}^d).
\]
By Lemma \ref{lem:Phid} the parameter array of $\Phi'$ is \eqref{eq:parrayd}.
By condition (ii) in Proposition \ref{prop:-A-As}, these parameter arrays coincide.
By this and Lemma \ref{lem:unique} 
$\Phi^{\downarrow\Downarrow}$ is isomorphic to $\Phi'$.
So $A,A^*$ is isomorphic to $-A,-A^*$.
\end{proofof}

\section{The case $d \leq 2$}
\label{sec:dleq2}

In this section we consider the case $d \leq 2$. 
In view of Lemma \ref{lem:char} we assume $\text{Char}(\F) \neq 2$.
The case $d=0$ is obvious, so we assume $d=1$ or $d=2$.
First consider the case $d=1$.

\begin{proposition}    \label{prop:d=1}    \samepage
\ifDRAFT {\rm prop:d=1}. \fi
For a nonzero $s \in \F$ with $s^2 \neq 1$, the pair
\begin{equation}                       \label{eq:TDTDd=1}
  \begin{pmatrix}
    0 & 1  \\
    1 & 0
  \end{pmatrix},
\qquad
  \begin{pmatrix}
    0 & s^{-1} \\
    s & 0
  \end{pmatrix}
\end{equation}
is a Leonard pair in $\Mat_2(\F)$.
Moreover, this Leonard pair has parameter array
\begin{equation}                                \label{eq:d1parray}
 \th_0 = 1, \; \th_1 = -1, \quad
 \th^*_0 = 1, \; \th^*_1 = -1,  \quad
 \vphi_1 = s+s^{-1}-2,  \quad
 \phi_1 =s + s^{-1}+2.
\end{equation}
\end{proposition}

\begin{proof}
One routinely checks that the sequence \eqref{eq:d1parray} is a parameter array over $\F$.
So there exists a Leonard pair $B,B^*$ that has parameter array \eqref{eq:d1parray}.
By Lemma \ref{lem:split} we may assume $B,B^*$ are as in \eqref{eq:splitAAs}:
\begin{align*}
 B &= 
  \begin{pmatrix}
    1 & 0  \\
    1 & -1
  \end{pmatrix},
&
 B^* &=
  \begin{pmatrix}
    1 & s+s^{-1}-2 \\
    0 & -1
  \end{pmatrix}.
\end{align*}
Define
\[
 P = 
   \begin{pmatrix}
      1 & s-1 \\
      s & 1-s
   \end{pmatrix}.
\]
Then $\text{\rm det} \, P = 1-s^2 \neq 0$, so $P$ is invertible.
One routinely checks that the pair $PBP^{-1}$, $PB^*P^{-1}$ coincides with
the pair \eqref{eq:TDTDd=1}.
So \eqref{eq:TDTDd=1} is a Leonard pair that is isomorphic to $B,B^*$.
The result follows.
\end{proof}

\begin{proposition}    \label{prop:d1exist}    \samepage
\ifDRAFT {\rm prop:d1exist}. \fi
Assume $d=1$.
Let $A,A^*$ be a Leonard pair on $V$ that is isomorphic to its opposite.
Then, after replacing $A,A^*$ with their scalar multiples if necessary, there exists a basis
for $V$ with respect to which the matrices representing $A,A^*$ are
as in Proposition \ref{prop:d=1}.
\end{proposition}

\begin{proof}
Let 
$(\{\th_i\}_{i=0}^d, \{\th^*_i\}_{i=0}^d, \{\vphi_i\}_{i=1}^d, \{\phi_i\}_{i=1}^d)$
be a parameter array of $A,A^*$.
Note that $\th_0 \neq 0$ and $\th^*_0 \neq 0$ by Lemma \ref{lem:thinonzero}.
By replacing $A,A^*$ with their scalar multiples, we may assume 
$\th_0 = 1$ and $\th^*_0 =1$.
By this and Proposition \ref{prop:-A-As},
$\th_d = -1$ and $\th^*_d = -1$.
Pick a nonzero $s \in \F$ such that $\vphi_1 = s + s^{-1} - 2$.
By Lemma \ref{lem:classify}(iv) $\phi_1 = s + s^{-1}+2$.
Therefore $A,A^*$ has parameter array as in \eqref{eq:d1parray}.
By this and Proposition \ref{prop:d=1} $A,A^*$ has the same parameter array
as the Leonard pair \eqref{eq:TDTDd=1}.
By this and Lemma \ref{lem:unique} $A,A^*$ is isomorphic to the Leonard pair \eqref{eq:TDTDd=1}.
The result follows.
\end{proof}

Theorem \ref{thm:main}(ii)$\Rightarrow$(i) for $d=1$ follows from Proposition \ref{prop:d1exist}.

\begin{proposition}    \label{prop:d1classify}    \samepage
\ifDRAFT {\rm prop:d1classify}. \fi
Let $A,A^*$ be a zero-diagonal TD-TD Leonard pair in $\Mat_2(\F)$.
Then, after replacing $A,A^*$ with their scalar multiples if necessary,
$A,A^*$ is equivalent to the Leonard pair in Proposition \ref{prop:d=1}
\end{proposition}

\begin{proof}
Let 
$(\{\th_i\}_{i=0}^d, \{\th^*_i\}_{i=0}^d, \{\vphi_i\}_{i=1}^d, \{\phi_i\}_{i=1}^d)$
be a parameter array of $A,A^*$.
By Theorem \ref{thm:main}(i)$\Rightarrow$(ii) $A,A^*$ is isomorphic to its opposite.
As in the proof of Proposition \ref{prop:d1exist} we may assume
$\th_0 = 1$, $\th_1 = -1$, $\th^*_0 = 1$, $\th^*_1 = -1$.
In view of Note \ref{note:subdiagonal1} we may assume $A,A^*$ take the form:
\begin{align*}
 A &= \begin{pmatrix}
         0 & z_1  \\
         1 & 0
        \end{pmatrix},
&
 A^* &= \begin{pmatrix}
             0 & y_1 z_1 \\
             x_1 & 0
           \end{pmatrix}
\end{align*}
for some nonzero scalars $x_1$, $y_1$, $z_1 \in \F$.
By Lemma \ref{lem:split} there exists a basis for $\F^2$,
with respect to which the matrices representing $A,A^*$ are
\begin{align*}
 B &= 
  \begin{pmatrix}
    1 & 0  \\
    1 & -1
  \end{pmatrix},
&
 B^* &=
  \begin{pmatrix}
    1 & \vphi_1 \\
    0 & -1
  \end{pmatrix}.
\end{align*}
By the construction, there exists an invertible matrix $P \in \Mat_2(\F)$
such that $AP=PB$ and $A^*P = P B^*$.
Compute the entries of $AP-PB$ and $A^*P-PB^*$ we obtain some equations.
Solving these equations, one finds that
$z_1 = 1$ and  $y_1 = x_1^{-1}$.
Now $A,A^*$ coincides with the pair \eqref{eq:TDTDd=1} by setting $s=x_1$.
\end{proof}

Next consider the case $d=2$.

\begin{proposition}   \label{prop:d2ex1}    \samepage
\ifDRAFT {\rm prop:d2ex1}. \fi
Let $y$, $z \in \F$ be nonzero scalars such that
\[
 y \neq 1,  \qquad
 y \neq -1, \qquad
 z \neq 1,  \qquad
 y z  \neq 1,  \qquad
 (y+1)z  \neq 2.
\]
Then the pair
\begin{equation}         \label{eq:d2ex1TDTD}
  \begin{pmatrix}
    0 & z & 0 \\
    1 & 0 & 1 - z  \\
    0 & 1 & 0
  \end{pmatrix},
\qquad
  \begin{pmatrix}
    0 & y z & 0 \\
    1 & 0 & y z -1  \\
    0 & -1 & 0
  \end{pmatrix}.
\end{equation}
is a Leonard pair in $\Mat_3(\F)$.
Moreover, this Leonard pair has parameter array
\begin{equation}                             \label{eq:d2ex1parray}
 \begin{matrix}        
   \th_0 = 1, \; \th_1 = 0, \; \th_2 = -1,  &
  \quad  \th^*_0 = 1, \; \th^*_1 = 0, \; \th^*_2 = -1,  
  \\  \rule{0mm}{4ex}
  \displaystyle \vphi_1 = \vphi_2 = (y+1)z-2,   &
  \quad \displaystyle \phi_1 = \phi_2 = (y+1)z.  
 \end{matrix}
\end{equation}
\end{proposition}

\begin{proof}
One routinely checks that the sequence \eqref{eq:d2ex1parray} is a parameter array over $\F$.
So there exists a Leonard pair $B,B^*$ that has parameter array \eqref{eq:d2ex1parray}.
By Lemma \ref{lem:split} we may assume $B,B^*$ are as in \eqref{eq:splitAAs}:
\begin{align*}
 B &= 
  \begin{pmatrix}
    1 & 0 & 0  \\
    1 & 0 & 0  \\
    0 & 1 & -1
  \end{pmatrix},
&
 B^* &=
  \begin{pmatrix}
    1 & (y+1)z-2 & 0 \\
    0 &  0 & (y+1)z-2 \\
    0 & 0 & -1
  \end{pmatrix}.
\end{align*}
Define
\[
 P = 
   \begin{pmatrix}
      y z & (1-y) z & (y+1) z^2 -2 z  \\
     1 & (y+1) z - 2 & 2 - (y+1) z  \\
    -1 & 2 & (y+1)z - 2
   \end{pmatrix}.
\]
One checks
\[
   \text{\rm det} \, P = (y+1)^2 z^2 \big( (y+1)z - 2 \big),
\]
so $P$ is invertible.
One routinely checks that the pair $PBP^{-1}$, $PB^*P^{-1}$ coincides with
the pair \eqref{eq:TDTDd=1}.
So \eqref{eq:TDTDd=1} is a Leonard pair that is isomorphic to $B,B^*$.
The result follows.
\end{proof}

\begin{proposition}    \label{prop:d2ex2}    \samepage
\ifDRAFT {\rm prop:d2ex2}. \fi
Let $s$, $t$, $z \in \F$ be nonzero scalars such that
\[
s^2 \neq 1,  \quad t^2 \neq 1, \quad s+t \neq 0, \quad z \neq 1.
\]
Define
\begin{align*}
 \tilde{y}_1 &= t z + \frac{1 - t^2}{s + t},
\\
 \tilde{y}_2 &= - s z + \frac{1 + s t}{s + t}.
\end{align*}
Then the pair
\begin{equation}                           \label{eq:d2ex2TDTD}
 \begin{pmatrix}
   0 & z & 0 \\
   1 & 0 & 1-z \\
   0 & 1 & 0
  \end{pmatrix},
\qquad
 \begin{pmatrix}
   0 & \tilde{y}_1 & 0 \\
   s & 0 & \tilde{y}_2 \\
   0 & t & 0
  \end{pmatrix}
\end{equation}
is a Leonard pair in $\Mat_3(\F)$.
Moreover, this Leonard pair has parameter array
\begin{equation}                                \label{eq:d2ex2parray}
 \begin{matrix}        
   \th_0 = 1, \; \th_1 = 0, \; \th_2 = -1,  & 
  \quad \th^*_0 = 1, \; \th^*_1 = 0, \; \th^*_2 = -1,  
  \\  \rule{0mm}{4ex}
  \displaystyle \vphi_1 = \vphi_2 = \frac{(s-1)(t-1)}{s+t},   &
  \quad \displaystyle \phi_1 = \phi_2 = \frac{(s+1)(t+1)}{s+t}.  
 \end{matrix}
\end{equation}
\end{proposition}

\begin{proof}
One routinely checks that the sequence \eqref{eq:d2ex2parray} is a parameter array over $\F$.
So there exists a Leonard pair $B,B^*$ that has parameter array \eqref{eq:d2ex2parray}.
By Lemma \ref{lem:split} we may assume $B,B^*$ are as in \eqref{eq:splitAAs}:
\begin{align*}
 B &= 
  \begin{pmatrix}
    1 & 0 & 0  \\
    1 & 0 & 0  \\
    0 & 1 & -1
  \end{pmatrix},
&
 B^* &=
  \begin{pmatrix}
    1 & \frac{(s-1)(t-1)}{s+t} & 0 \\
    0 &  0 & \frac{(s-1)(t-1)}{s+t} \\
    0 & 0 & -1
  \end{pmatrix}.
\end{align*}
Define
\[
 P = 
   \begin{pmatrix}
      1-t^2+(s+t)tz & t^2-1-(s+t)(t-1)z & (s-1)(t-1)z \\
      s+t & (s-1)(t-1) &  (s-1)(1-t)  \\
     (s+t)t & (s+t)(1-t) & (s-1)(t-1)
   \end{pmatrix}.
\]
One checks
\[
   \text{\rm det}\, P = (1-s^2)(t^2-1)^2,
\]
so $P$ is invertible.
One routinely checks that the pair $PBP^{-1}$, $PB^*P^{-1}$ coincides with
the pair \eqref{eq:TDTDd=1}.
So \eqref{eq:TDTDd=1} is a Leonard pair that is isomorphic to $B,B^*$.
The result follows.
\end{proof}

\begin{proposition}    \label{prop:d2exist}    \samepage
\ifDRAFT {\rm prop:d2exist}. \fi
Assume $d=2$.
Let $A,A^*$ be a Leonard pair on $V$ that is isomorphic to its opposite.
Then, after replacing $A,A^*$ with their scalar multiples if necessary, there exists a basis
for $V$ with respect to which the matrices representing $A,A^*$ are
as in Proposition \ref{prop:d2ex1}.
\end{proposition}

\begin{proof}
Let 
$(\{\th_i\}_{i=0}^d, \{\th^*_i\}_{i=0}^d, \{\vphi_i\}_{i=1}^d, \{\phi_i\}_{i=1}^d)$
be a parameter array of $A,A^*$.
Note that $\th_0 \neq 0$, $\th^*_0 \neq 0$ by Lemma \ref{lem:thinonzero}.
By replacing $A,A^*$ with their scalar multiples, we may assume 
$\th_0 = 1$ and $\th^*_0 =1$.
By this and Proposition \ref{prop:-A-As},
$\th_1 = 0$, $\th_2 = -1$, $\th^*_1 = 0$, $\th^*_2 = -1$.
Pick nonzero $y,z \in \F$ such that $\vphi_1 = (y+1) z - 2$.
By Lemma \ref{lem:classify}(iv) $\phi_1 = (y+1) z$.
Therefore $A,A^*$ has parameter array as in \eqref{eq:d2ex1parray}.
By this and Proposition \ref{prop:d2ex1} $A,A^*$ has the same parameter array
as the Leonard pair \eqref{eq:d2ex1TDTD}.
By this and Lemma \ref{lem:unique} $A,A^*$ is isomorphic to the Leonard pair \eqref{eq:d2ex1TDTD}.
The result follows.
\end{proof}

Theorem \ref{thm:main}(ii)$\Rightarrow$(i) for $d=2$ follows from Proposition \ref{prop:d2exist}.

\begin{proposition}    \label{prop:d2classify}    \samepage
\ifDRAFT {\rm prop:d2classify}. \fi
Let $A,A^*$ be a zero-diagonal TD-TD Leonard pair in $\Mat_3(\F)$.
Then, after replacing $A,A^*$ with their scalar multiples if necessary,
$A,A^*$ or its anti-diagonal transpose is equivalent to the Leonard pair \eqref{eq:d2ex1TDTD} or 
\eqref{eq:d2ex2TDTD}.
\end{proposition}

\begin{proof}
Let 
$(\{\th_i\}_{i=0}^d, \{\th^*_i\}_{i=0}^d, \{\vphi_i\}_{i=1}^d, \{\phi_i\}_{i=1}^d)$
be a parameter array of $A,A^*$.
By Theorem \ref{thm:main}(i)$\Rightarrow$(ii) $A,A^*$ is isomorphic to its opposite.
As in the proof of Proposition \ref{prop:d2exist} we may assume
$\th_0 = 1$, $\th_1 = 0$, $\th_2=-1$, $\th^*_0 = 1$, $\th^*_1 = 0$, $\th^*_2 = -1$.
By Lemma \ref{lem:classify}(iii), (iv),
\begin{align}
  \vphi_2 &= \vphi_1,  &
  \phi_1 &= \vphi_1 + 2, &
  \phi_2 &= \vphi_1 + 2.
\end{align}                                         \label{eq:d2vphi}
In view of Note \ref{note:subdiagonal1} we may assume $A,A^*$ take the form:
\begin{align*}
 A &= \begin{pmatrix}
         0 & z_1 & 0 \\
         1 & 0 & z_2 \\
         0 & 1 & 0
        \end{pmatrix},
&
 A^* &= \begin{pmatrix}
             0 & y_1 z_1 & 0 \\
             x_1 & 0 & y_2 z_2 \\
              0 & x_2 & 0
           \end{pmatrix}
\end{align*}
for some nonzero scalars $x_1$, $x_2$, $y_1$, $y_2$,  $z_1$, $z_2 \in \F$.
By Lemma \ref{lem:split} there exists a basis for $\F^3$,
with respect to which the matrices representing $A,A^*$ are
\begin{align*}
 B &= 
  \begin{pmatrix}
    1 & 0  & 0 \\
    1 & 0  & 0  \\
    0 & 1 & -1
  \end{pmatrix},
&
 B^* &=
  \begin{pmatrix}
    1 & \vphi_1 & 0 \\
    0 & 0 & \vphi_1 \\
    0 & 0 & -1
  \end{pmatrix}.
\end{align*}
By the construction, there exists an invertible matrix $P \in \Mat_3(\F)$
such that $AP=PB$ and $A^*P = P B^*$.
We compute the entries of $AP=PB$ and $A^*P= PB^*$ as follows.
In $A^*P=PB^*$, compute the $(0,0)$ and $(2,0)$ entry to find that
\begin{align*}
 P_{1,0} &= \frac{ P_{0,0} }
                    { y_1 z_1 },  
&
 P_{2,0} &= \frac{ P_{0,0} (1 - x_1 y_1 z_1)}
                      {y_1 y_2 z_1 z_2}. 
\end{align*}
Observe that $P_{0,0} \neq 0$; otherwise the $0$th column of $P$ is $0$,
contradicting that $P$ is invertible.
By replacing $P$ with $P_{0,0}^{-1} P$, we may assume $P_{0,0}=1$.
So
\begin{align*}
 P_{1,0} &= \frac{1}{y_1 z_1},    
&
 P_{2,0} &= \frac{1 - x_1 y_1 z_1}
                      {y_1 y_2 z_1 z_2}. 
\end{align*}
In $AP=PB$, compute the $(0,0)$, $(0,1)$, $(1,1)$,  $(2,2)$ entries,
and in $A^*P=PB^*$, compute the $(0,1)$, $(0,2)$ entries to find that
\begin{align*}
 P_{0,1} &= \frac{1}{y_1} - 1, 
&
 P_{0,2} &= \frac{\vphi_1}{y_1},
\\
 P_{1,1} &= \frac{\vphi_1}{y_1 z_1},
&
 P_{1,2} &= - \frac{\vphi_1}{y_1 z_1},
\\
 P_{2,1} &= - \frac{\vphi_1 + z_1 - y_1 z_1}{y_1 z_1 z_2},
&
 P_{2,2} &= \frac{\vphi_1}{y_1 z_1}.
\end{align*}
By $(1,2)$-entry of $AP=PB$,
\[
   \vphi_1 (z_1 + z_2 - 1) =0.
\]
By this and $\vphi_1 \neq 0$,
\begin{equation}
   z_2 = 1 - z_1.                      \label{eq:d2z2}
\end{equation}
Note that $z_1 \neq 1$; otherwise $z_2=0$.
By $(2,0)$-entry of $A^*P=PB^*$,
\[
    1 - x_1 y_1 z_1 + x_2 y_2 (z_1 - 1) = 0.
\]
So
\begin{equation}
   y_2 = \frac{ x_1 y_1 z_1 - 1}{x_2 (z_1 - 1)}.    \label{eq:d2y2}
\end{equation}
By $(1,0)$-entry of $AP=PB$,
\[
  - \vphi_1  + x_2 - x_2 z_1 + y_1 z_1 -1 = 0.
\]
So
\begin{equation}
 \vphi_1 = x_2-x_2 z_1 + y_1 z_1 - 1.              \label{eq:d2vphi1}
\end{equation}
By $(1,1)$-entry of $A^*P=PB^*$,
\begin{equation}
  - y_1 z_1 (x_1 + x_2 ) + x_2^2 (z_1-1)+ x_1 x_2 z_1 + 1 = 0.    \label{eq:d2equat}
\end{equation}

First assume $x_1 + x_2 = 0$.
Then \eqref{eq:d2equat} becomes $x_1^2 = 1$.
So either $x_1 = 1$ or $x_1 = -1$.
If $x_1 = 1$, then $y_2 z_2 = y_1 z_1 -1$, and so
$A,A^*$ coincides with \eqref{eq:d2ex1TDTD} with $y=y_1$ and $z = z_1$.
If $x_1 = -1$, then $y_2 z_2 = y_1 z_1 + 1$, and so 
\begin{align*}
 A &= 
  \begin{pmatrix}
     0 & z_1 & 0 \\
     1 & 0 & 1-z_1 \\
     0 & 1 & 0
  \end{pmatrix},
&
  A^* &=
  \begin{pmatrix}
    0 & y_1 z_1 & 0  \\
   -1 & 0 & y_1 z_1 + 1  \\
    0 & 1 & 0
  \end{pmatrix}.
\end{align*}
Setting
\begin{align*}
  z &= 1 - z_1,  &
  y &= \frac{y_1z_1 + 1}{1 - z_1},
\end{align*}
the above matrices become
\begin{align*}
 A &= 
  \begin{pmatrix}
     0 & 1-z & 0 \\
     1 & 0 & z \\
     0 & 1 & 0
  \end{pmatrix},
&
  A^* &=
  \begin{pmatrix}
    0 & y z -1 & 0  \\
   -1 & 0 & y z  \\
    0 & 1 & 0
  \end{pmatrix}.
\end{align*}
This coincides with the anti-diagonal transpose of \eqref{eq:d2ex1TDTD}.

Next assume $x_1 + x_2 \neq 0$.
By \eqref{eq:d2equat}
\[
  y_1 z_1  = x_2 z_1 + \frac{1-x_2^2 }{x_1 + x_2}. 
\]
By this and \eqref{eq:d2y2}
\[
  y_2 z_2 = - x_1 z_1 + \frac{ 1 + x_1 x_2}{x_1 + x_2}.
\]
Now $A,A^*$ coincides with the pair \eqref{eq:d2ex2TDTD} by setting 
$s=x_1$, $t = x_2$, $z = z_1$.
The result follows.
\end{proof}

\section{Parameter arrays in closed form}
\label{sec:types}

For the rest of the paper we assume $d \geq 3$.
In this section we recall the formulas that represent the parameter array
in closed form.
In view of Lemma \ref{lem:char}, we assume $\text{Char}(\F) \neq 2$.
Let $A,A^*$ be a Leonard pair on $V$ with parameter array 
\[
    (\{\th_i\}_{i=0}^d, \{\th^*_i\}_{i=0}^d, \{\vphi_i\}_{i=1}^d, \{\phi_i\}_{i=1}^d),
\]
and let $\beta$ be the fundamental parameter of $A,A^*$.

\begin{lemma} {\rm  (See \cite[Lemma 14.1]{NT:affine}.)  }    \label{lem:typeIIclosed}   \samepage
\ifDRAFT {\rm lem:typeIIclosed}. \fi
Assume $\beta=2$.
Then there exist scalars $\alpha$, $h$, $\mu$, $\alpha^*$, $h^*$, $\mu^*$, $\tau$ in $\F$
such that
\begin{align*}
 \th_i &= \alpha + \mu (i-d/2) + h i(d-i),
\\
 \th^*_i &= \alpha^* + \mu^* (i-d/2) + h^* i(d-i)
\intertext{for $0 \leq i \leq d$, and}
 \vphi_i &= i(d-i+1)(\tau- \mu\mu^*/2+(h \mu^* + \mu h^*)(i-(d+1)/2)+h h^*(i-1)(d-i)),
\\
 \phi_i &= i(d-i+1)(\tau+ \mu\mu^*/2+(h \mu^* - \mu h^*)(i-(d+1)/2)+h h^*(i-1)(d-i))
\end{align*}
for $1 \leq i \leq d$.
\end{lemma}

\begin{note}  {\rm (See \cite[Remark 14.2]{NT:affine}.) }  \label{note:typeIIclosed}  \samepage
\ifDRAFT {\rm note:typeIIclosed}. \fi
Referring to Lemma \ref{lem:typeIIclosed},
$\text{Char}(\F)$ is $0$ or greater than $d$.
\end{note}

\begin{lemma} {\rm  (See \cite[Lemma 15.1]{NT:affine}.) }   \label{lem:typeIII+closed}  \samepage
\ifDRAFT {\rm lem:typeIII+closed}. \fi
Assume $\beta=-2$ and $d$ is even.
Then there exist scalars $\alpha$, $h$, $\sigma$, $\alpha^*$, $h^*$, $\sigma^*$, $\tau$ in $\F$
such that
\begin{align*}
 \th_i &=
   \begin{cases}
       \alpha + \sigma + h (i-d/2) & \text{ if $i$ is even},
   \\
       \alpha - \sigma - h(i-d/2) & \text{ if $i$ is odd},
   \end{cases}
\\
 \th^*_i &=
   \begin{cases}
     \alpha^* + \sigma^* + h^* (i-d/2) & \text{ if $i$ is even},
   \\
      \alpha^* - \sigma^*  - h^* (i-d/2) & \text{ if $i$ is odd}
   \end{cases}
\intertext{for $0 \leq i \leq d$, and}
 \vphi_i &=
  \begin{cases}
    i \big( \tau- \sigma h^* - \sigma^* h - h h^* (i-(d+1)/2) \big) & \text{ if $i$ is even},
  \\
   (d-i+1) \big( \tau + \sigma h^* + \sigma^* h + h h^* (i-(d+1)/2) \big) & \text{ if $i$ is odd},
  \end{cases}
\\
 \phi_i &=
  \begin{cases}
    i \big( \tau- \sigma h^* + \sigma^* h + h h^* (i-(d+1)/2) \big) & \text{ if $i$ is even},
  \\
   (d-i+1) \big( \tau + \sigma h^* - \sigma^* h - h h^* (i-(d+1)/2) \big) & \text{ if $i$ is odd}
  \end{cases}
\end{align*}
for $1 \leq i \leq d$.
\end{lemma}

\begin{note} {\rm (See \cite[Remark 15.2]{NT:affine}.) }  \label{note:typeIIIclosed}  \samepage
\ifDRAFT {\rm note:typeIIIclosed}. \fi
Referring to Lemma \ref{lem:typeIII+closed},
$\text{Char}(\F)$ is either $0$ or greater than $d/2$.
\end{note}

\begin{lemma}  {\rm  (See \cite[Lemma 13.1]{NT:affine}.)  }    \label{lem:typeIclosed}
\ifDRAFT {\rm lem:typeIclosed}. \fi
Assume $\beta \neq 2$ and $\beta \neq -2$.
Pick a nonzero $q \in \F$ such that $\beta = q+q^{-1}$.
Then there exist scalars $\alpha$, $h$, $\mu$, $\alpha^*$, $h^*$, $\mu^*$, $\tau$ in $\F$
such that
\begin{align*}
 \th_i &= \alpha + \mu q^i + h q^{d-i},
\\
 \th^*_i &= \alpha^* + \mu^* q^i + h^* q^{d-i}
\intertext{for $0 \leq i \leq d$, and}
 \vphi_i &= (q^i-1)(q^{d-i+1}-1)(\tau - \mu \mu^* q^{i-1} - h h^* q^{d-i}),
\\
 \phi_i &= (q^i-1)(q^{d-i+1}-1)(\tau - h \mu^* q^{i-1} - \mu h^* q^{d-i})
\end{align*}
for $1 \leq i \leq d$.
\end{lemma}

\begin{note}  {\rm (See \cite[Remark 13.2]{NT:affine}.) }  \label{note:typeIclosed}  \samepage
\ifDRAFT {\rm note:typeIclosed}. \fi
Referring to Lemma \ref{lem:typeIclosed},
$q^i \neq 1$ for $1 \leq i \leq d$.
\end{note}

\section{The parameter array of  a Leonard pair that is isomorphic to its opposite}
\label{sec:parray}

Let $A,A^*$ be a Leonard pair on $V$ that is isomorphic to its opposite.
Let $\beta$ be the fundamental parameter and let
\begin{equation}
  (\{\th_i\}_{i=0}^d, \{\th^*_i\}_{i=0}^d, \{\vphi_i\}_{i=1}^d, \{\phi_i\}_{i=1}^d)        \label{eq:parray0}
\end{equation}
be a parameter array of $A,A^*$.
Note that $\text{Char}(\F) \neq 2$ by Lemma \ref{lem:char}.
Also note by Lemma \ref{lem:PhidD} 
that $\th_i + \th_{d-i}=0$ and $\th^*_i + \th^*_{d-i}=0$ for $0 \leq i \leq d$.

\begin{proposition}    \label{prop:type2closed}   \samepage
\ifDRAFT {\rm prop:type2closed}. \fi
Assume $\beta = 2$.
Then there exists a nonzero scalar $s \in \F$ such that
\begin{align}
 \th_i &= d - 2i  && (0 \leq i \leq d),                     \label{eq:type2th}
\\
 \th^*_i &= d-2i && (0 \leq i \leq d),                     \label{eq:type2ths}
\\
 \vphi_i &= i(d-i+1)(s + s^{-1} - 2)  && (1 \leq i \leq d),      \label{eq:type2vphi}
\\
 \phi_i &= i(d-i+1)(s + s^{-1} + 2)   && (1 \leq i \leq d),     \label{eq:type2phi}
\end{align}
after replacing $A,A^*$ with their nonzero scalar multiples if necessary.
\end{proposition}

\begin{proof}
Let the scalars $\alpha$, $h$, $\mu$, $\alpha^*$, $h^*$, $\mu^*$, $\tau$
be from Lemma \ref{lem:typeIIclosed}.
Observe
\begin{align*}
\th_0 &= \alpha + \mu (-d/2),  & \th_d &= \alpha + \mu (d-d/2).
\end{align*}
By this and $\th_0 + \th_d =0$ we find $2 \alpha=0$. This forces $\alpha=0$
since $\text{Char}(\F) \neq 2$.
Observe 
\begin{align*}
  \th_1 &= \mu(1-d/2) + h(d-1),  & \th_{d-1} &= \mu (d-1-d/2) + h(d-1).
\end{align*}
By this and $\th_1 + \th_{d-1}=0$ we find $2h(d-1) = 0$.
This forces $h=0$ by Note \ref{note:typeIIclosed}.
By these comments,
$\th_i = \mu (i-d/2)$ for $0 \leq i \leq d$.
By replacing $A$ with its nonzero scalar multiple if necessary, we may assume $\mu=-2$.
So \eqref{eq:type2th} holds.
Similarly, $\alpha^* = 0$ and $h^*=0$, and we may assume $\mu^* = -2$.
So \eqref{eq:type2ths} holds.
Pick a nonzero $s \in \F$ such that $\tau = s + s^{-1}$.
Then \eqref{eq:type2vphi} and \eqref{eq:type2phi} hold.
\end{proof}

\begin{lemma}   \label{lem:type2cond}   \samepage
\ifDRAFT {\rm lem:type2cond}. \fi
For a nonzero $s \in \F$, define scalars 
$\{\th_i\}_{i=0}^d$, $\{\th^*_i\}_{i=0}^d$, $\{\vphi_i\}_{i=1}^d$, $\{\phi_i\}_{i=1}^d$
by \eqref{eq:type2th}--\eqref{eq:type2phi}.
Then \eqref{eq:parray0} is a parameter array over $\F$
if and only if the following {\rm (i)} and {\rm (ii)} hold:
\begin{itemize}
\item[\rm (i)]
 $\text{\rm Char}(\F)$ is either $0$ or greater than $d$.
\item[\rm (ii)]
$s^2 \neq 1$.
\end{itemize}
\end{lemma}

\begin{proof}
First assume \eqref{eq:parray0} is a parameter array over $\F$.

(i):
See Note \ref{note:typeIIclosed}.

(ii):
If $s^2=1$, then $\vphi_1=0$ or $\phi_1=0$,
contradicting  Lemma \ref{lem:classify}(ii).

Next assume (i) and (ii) hold.
One routinely checks conditions (i)--(v) in Lemma \ref{lem:classify}.
So \eqref{eq:parray0} is a parameter array over $\F$.
\end{proof}

\begin{proposition}    \label{prop:type3closed} \samepage
\ifDRAFT {\rm prop:type3closed}. \fi
Assume $\beta = -2$.
Then $d$ is even.
Moreover, there exists a  scalar $\tau \in \F$ such that
\begin{align}
 \th_i &=
   \begin{cases}
     2i-d & \text{ if $i$ is even},
   \\
    d-2i &  \text{ if $i$ is odd}
  \end{cases}          && (0 \leq i \leq d),       \label{eq:type3th}
\\
 \th^*_i &=
   \begin{cases}
     2i-d & \text{ if $i$ is even},
   \\
    d-2i &  \text{ if $i$ is odd},
  \end{cases}            && (0 \leq i \leq d),    \label{eq:type3ths}
\\
 \vphi_i &= 
  \begin{cases}
    2i (d-2i+1+ \tau)   &  \text{ if $i$ is even},
  \\
   -2(d-i+1)(d-2i+1-\tau)  & \text{ if $i$ is odd}
  \end{cases}              && (1 \leq i \leq d),    \label{eq:type3vphi}
\\
 \phi_i &= 
   \begin{cases}
     -2i (d-2i+1- \tau)   & \text{ if $i$ is even},
   \\
    2(d-i+1)(d-2i+1+ \tau)  & \text{ if $i$ is odd}
   \end{cases}                        &&   (1 \leq i \leq d),  \label{eq:type3phi}
\end{align}
after replacing $A,A^*$ with their nonzero scalar multiples if necessary.
\end{proposition}

\begin{proof}
We first show that $d$ is even.
By way of contradiction, assume $d$ is odd.
Set $m=(d-1)/2$.
By Definition \ref{def:beta},
\[
   \frac{\th_{m-1} - \th_{m+2}}
          {\th_m - \th_{m+1}}   
   = \beta + 1 = -1.
\]
We have $\th_m + \th_{m+1} =0$ and $\th_{m-1} + \th_{m+2} = 0$.
By these comments $\th_m = \th_{m+2}$, contradicting Lemma \ref{lem:classify}(i).
Thus $d$ must be even.
Let the scalars $\alpha$, $h$, $\sigma$, $\alpha^*$, $h^*$, $\sigma^*$, $\tau$
be from Lemma \ref{lem:typeIII+closed}.
Observe
\begin{align*}
  \th_0 &= \alpha + \sigma + h(-d/2),  &
  \th_d &= \alpha + \sigma + h(d-d/2).
\end{align*}
By this and $\th_0 + \th_d=0$ we find $2 (\alpha+\sigma)=0$.
This forces $\alpha+\sigma=0$ by $\text{Char}(\F) \neq 2$.
Observe
\begin{align*}
  \th_1 &= \alpha - \sigma - h(1-d/2), &
 \th_{d-1} &= \alpha - \sigma - h(d-1-d/2).
\end{align*}
By this and $\th_1 + \th_{d-1} = 0$ we find $2(\alpha - \sigma)=0$, so $\alpha - \sigma=0$.
By these comments $\alpha = 0$ and $\sigma=0$.
So $\th_i = h (i-d/2)$ if $i$ is even, and $\th_i = -h(i-d/2)$ if $i$ is odd.
By replacing $A$ with its scalar multiple if necessary, we may assume $h=2$.
Similarly $\alpha^* =0$ and $\sigma^*=0$, and we may assume $h^* = 2$.
So \eqref{eq:type3th} and \eqref{eq:type3ths} hold.
Replacing $\tau$ with $2 \tau$ we get \eqref{eq:type3vphi} and \eqref{eq:type3phi}.
\end{proof}

\begin{lemma}   \label{lem:type3cond}   \samepage
\ifDRAFT {\rm lem:type3cond}. \fi
Assume $d$ is even.
For  $\tau \in \F$, define scalars 
$\{\th_i\}_{i=0}^d$, $\{\th^*_i\}_{i=0}^d$, $\{\vphi_i\}_{i=1}^d$, $\{\phi_i\}_{i=1}^d$
by \eqref{eq:type3th}--\eqref{eq:type3phi}.
Then \eqref{eq:parray0} is a parameter array over $\F$
if and only if the following {\rm (i)} and {\rm (ii)} hold:
\begin{itemize}
\item[\rm (i)]
$\text{\rm Char}(\F)$ is $0$ or greater than $d$.
\item[\rm (ii)]
$\tau$ is not among $1-d$, $3-d$, \ldots, $d-1$.
\end{itemize}
\end{lemma}

\begin{proof}
First assume \eqref{eq:parray0} is a parameter array over $\F$.

(i):
Follows from the fact that $\{\th_i\}_{i=0}^d$ are mutually distinct.

(ii):
By way of contradiction, assume that $\tau$ is among
$1-d$, $3-d$, \ldots, $d-1$.
So $\tau$ is an odd integer such that $1-d \leq \tau \leq d-1$.
Set $i = (d+1-\tau)/2$.
Observe $d-2i+1-\tau=0$, and $i$ is an integer such that $1 \leq i \leq d$.
Now $\phi_i = 0$ by \eqref{eq:type3phi} if $i$ is even,
and $\vphi_i=0$ by \eqref{eq:type3vphi} if $i$ is odd;
contradicting Lemma \ref{lem:classify}(ii).

Next assume (i) and (ii) hold.
One routinely checks conditions (i)--(v) in Lemma \ref{lem:classify}.
So \eqref{eq:parray0} is a parameter array over $\F$.
\end{proof}

\begin{proposition}    \label{prop:type1closed}   \samepage
\ifDRAFT {\rm prop:type1closed}. \fi
Assume $\beta \neq 2$ and $\beta \neq -2$.
Then there exist nonzero scalars $q$, $s \in \F$ such that
\begin{align}
 \th_i &=  q^i - q^{d-i}  && (0 \leq i \leq d),                     \label{eq:type1th}
\\
 \th^*_i &=  q^i - q^{d-i}   && (0 \leq i \leq d),                     \label{eq:type1ths}
\\
 \vphi_i &=   (q^i-1)(q^{d-i+1}-1)(s-q^{i-1})(s-q^{d-i}) s^{-1}  
                                              && (1 \leq i \leq d),      \label{eq:type1vphi}
\\
 \phi_i &=  (q^i-1)(q^{d-i+1}-1)(s+q^{i-1})(s+q^{d-i}) s^{-1}
                                              && (1 \leq i \leq d).      \label{eq:type1phi}
\end{align}
The scalar $q$ satisfies $\beta = q+q^{-1}$.
\end{proposition}

\begin{proof}
Let the scalars $\alpha$, $h$, $\mu$, $\alpha^*$, $h^*$, $\mu^*$, $\tau$
be from Lemma \ref{lem:typeIclosed}.
Observe
\begin{align*}
 \th_0 &= \alpha + \mu + h q^d, &
 \th_d &= \alpha + \mu q^d + h.
\end{align*}
By this and $\th_0 + \th_d=0$,
\begin{equation}
  2 \alpha + (\mu + h)(q^d + 1) = 0.                  \label{eq:type1closedaux1}
\end{equation}
Observe
\begin{align*}
 \th_1 &= \alpha + \mu q + h q^{d-1}, &
 \th_{d-1} &= \alpha + \mu q^{d-1} + h q.
\end{align*}
By this and $\th_1 + \th_{d-1}=0$,
\begin{equation}
  2 \alpha + (\mu + h)(q + q^{d-1}) = 0.             \label{eq:type1closedaux2}
\end{equation}
In \eqref{eq:type1closedaux1} and \eqref{eq:type1closedaux2}, eliminate $\alpha$ to find
\[
   (q-1)(q^{d-1}-1)(\mu + h)  = 0.
\]
By this and Note \ref{note:typeIclosed} $\mu+h=0$.
By this and \eqref{eq:type1closedaux2} $\alpha = 0$.
By replacing $A$ with its nonzero scalar multiple if necessary,
we may assume $\mu=1$, and so $h=-1$.
Similarly, $\alpha^*=0$ and $\mu^* + h^* =0$,
and we may assume $\mu^* = 1$ and $h^* = -1$.
Pick a nonzero $s \in \F$ such that $\tau = s + s^{-1}q^{d-1}$.
Then \eqref{eq:type1th}--\eqref{eq:type1phi} hold.
By \eqref{eq:type1th} and Definition \ref{def:beta} one finds $\beta = q+q^{-1}$.
\end{proof}

\begin{note}   \label{note:s}   \samepage
\ifDRAFT {\rm note:s}. \fi
In Proposition \ref{prop:type1closed}, the scalar $s$ can be replaced by $s^{-1} q^{d-1}$.
Actually, if we replace $s$ with $s^{-1} q^{d-1}$, the values of
\eqref{eq:type1vphi} and \eqref{eq:type1phi} are invariant.
\end{note}

\begin{lemma}   \label{lem:type1cond}   \samepage
\ifDRAFT {\rm lem:type1cond}. \fi
For nonzero $q$, $s \in \F$, define scalars 
$\{\th_i\}_{i=0}^d$, $\{\th^*_i\}_{i=0}^d$, $\{\vphi_i\}_{i=1}^d$, $\{\phi_i\}_{i=1}^d$
by \eqref{eq:type1th}--\eqref{eq:type1phi}.
Then \eqref{eq:parray0} is a parameter array over $\F$
if and only if the following {\rm (i)--(iii)} hold:
\begin{itemize}
\item[\rm (i)]
$q^i \neq 1$ for $1 \leq i \leq d$;
\item[\rm (ii)]
$q^i \neq -1$ for $0 \leq i \leq d-1$.
\item[\rm (iii)]
$s^2 \neq q^{2i}$ for $0 \leq i \leq d-1$.
\end{itemize}
\end{lemma}

\begin{proof}
First assume \eqref{eq:parray0} is a parameter array over $\F$.

(i): 
See Note \ref{note:typeIclosed}.

(ii):
Assume $q^i=-1$ for some $i$ $(0 \leq i \leq d-1)$.
Then  $\th_0 - \th_{d-i} = (q^i+1)(1-q^{d-i})=0$ by \eqref{eq:type1th}, 
contradicting Lemma \ref{lem:classify}(i).

(iii):
Assume $s^2 = q^{2i}$ for some $i$ $(0 \leq i \leq d-1)$.
So $s = q^i$ or $s = - q^i$.
First assume $s=q^i$.
Then $\vphi_{i+1}=0$ by \eqref{eq:type1vphi},
contradicting Lemma \ref{lem:classify}(ii).
Next assume $s = -q^{i}$.
Then $\phi_{i+1}=0$ by \eqref{eq:type1phi}, contradicting Lemma \ref{lem:classify}(ii).
The result follows.

Next assume (i)--(iii) hold.
One routinely checks conditions (i)--(v) in Lemma \ref{lem:classify}.
So \eqref{eq:parray0} is a parameter array over $\F$.
\end{proof}

\begin{definition}         \label{def:types}    \samepage
\ifDRAFT {\rm def:types}. \fi
We define the type of a Leonard pair as follows.
\begin{itemize}
\item[(i)]
$A,A^*$ is said to have {\em Krawtchouk type} whenever
$\beta=2$, $h=0$ and $h^*=0$, where $h,h^*$ are from Lemma \ref{lem:typeIIclosed}.
\item[(ii)]
$A,A^*$ is said to have {\em Bannai/Ito type} whenever $\beta = -2$.
\item[(iii)]
$A,A^*$ is said to have {\em $q$-Racah type}
whenever $\mu \neq 0$, $h \neq 0$, $\mu^* \neq 0$ and $h \neq 0$,
where $\mu, h, \mu^*, h^*$ are from Lemma \ref{lem:typeIclosed}.
\end{itemize}
\end{definition}

\begin{proofof}{Proposition \ref{prop:types}}
Let $A,A^*$ be a Leonard pair on $V$ that is isomorphic to $-A,-A^*$.
Let $\beta$ be the fundamental parameter of $A,A^*$.
First assume $\beta = 2$.
Then $A,A^*$ has Krawtchouk type by Proposition \ref{prop:type2closed}.
Next assume $\beta = -2$.
Then $A,A^*$ has Bannai/Ito type with even diameter by
Proposition \ref{prop:type3closed}.
Next assume $\beta \neq 2$ and $\beta \neq -2$.
Then $A,A^*$ has $q$-Racah type by Proposition \ref{prop:type1closed}.
\end{proofof}

\section{List of zero-diagonal TD-TD Leonard pairs in $\Mat_{d+1}(\F)$}
\label{sec:list}

In this section, we display five families of zero-diagonal TD-TD Leonard pairs in $\Mat_{d+1}(\F)$.
In view of Note \ref{note:subdiagonal1}, for nonzero scalars $\{x_i\}_{i=1}^d$, $\{y_i\}_{i=1}^d$, $\{z_i\}_{i=1}^d$,
we consider the following zero-diagonal TD-TD pair in $\Mat_{d+1}(\F)$:
\begin{align}           \label{eq:TDTD}
&
 \begin{pmatrix}
  0  & z_1 & & & & \text{\bf 0} \\
  1 & 0  & z_2 \\
   & 1 & 0 & \cdot \\
   & & \cdot & \cdot & \cdot \\
   & & & \cdot & \cdot & z_d \\
  \text{\bf 0}  & & & & 1 & 0
 \end{pmatrix},
& 
 &
 \begin{pmatrix}
   0 & \yb_1 & & & & \text{\bf 0} \\
  x_1 & 0 & \yb_2  \\
   & x_2 & 0 & \cdot  \\
   & & \cdot & \cdot & \cdot \\
   & & & \cdot & \cdot & \yb_{d} \\
  \text{\bf 0} & & & & x_d & 0
 \end{pmatrix},
\end{align}
where $\yb_i = y_i z_i$ for $1 \leq i \leq d$.

\begin{proposition}   \label{prop:type2ex}   \samepage
\ifDRAFT {\rm prop:type2ex}. \fi
Fix a nonzero $s \in \F$.
Assume the conditions {\rm (i), (ii)} in Lemma \ref{lem:type2cond} hold.
Consider the pair \eqref{eq:TDTD} with
\begin{align*}
 x_i &= s   && (1 \leq i \leq d),    \\
 y_i &= s^{-1}  && (1 \leq i \leq d),   \\
 z_i &= i (d-i+1)  &&  (1 \leq i \leq d).
\end{align*}
Then \eqref{eq:TDTD} is a Leonard pair in $\Mat_{d+1}(\F)$.
Moreover, this Leonard pair has fundamental parameter $\beta=2$ and 
parameter array in Proposition \ref{prop:type2closed}.
\end{proposition}

\begin{proposition}   \label{prop:type3ex}   \samepage
\ifDRAFT {\rm prop:type3ex}. \fi
Fix $\tau \in \F$ and $\epsilon \in \{1,-1\}$.
Assume the conditions {\rm (i), (ii)} in Lemma \ref{lem:type3cond} hold.
Consider the pair \eqref{eq:TDTD} with
\begin{align*}
 x_i &= (-1)^{i-1}  \epsilon && (1 \leq i \leq d),    \\
 y_i &= (-1)^{i-1} \epsilon && (1 \leq i \leq d),   \\
 z_i &= \begin{cases}
            i (d-i+1- \epsilon \tau) & \text{ if $i$ is even},  \\
           (d-i+1)(i + \epsilon \tau)  &  \text{ if $i$ is odd}
         \end{cases}
                 &&  (1 \leq i \leq d).
\end{align*}
Then \eqref{eq:TDTD} is a Leonard pair in $\Mat_{d+1}(\F)$.
Moreover, this Leonard pair has fundamental parameter $\beta=-2$ and 
parameter array in Proposition \ref{prop:type3closed}.
\end{proposition}

\begin{proposition}   \label{prop:type1compact}   \samepage
\ifDRAFT {\rm prop:type1compact}. \fi
Fix nonzero scalars $q, s \in \F$.
Assume the conditions {\rm (i)--(iii)} in Lemma \ref{lem:type1cond} hold.
Consider the pair \eqref{eq:TDTD} with
\begin{align*}
 x_i &= s q^{1-i}                        &&  (1 \leq i \leq d),
\\
 y_i &= s^{-1} q^{d-i}                 &&  (1 \leq i \leq d),
\\
 z_i &= q^{i-1} (q^i-1)(q^{d-i+1}-1).   &&  (1 \leq i \leq d).
\end{align*}
Then \eqref{eq:TDTD} is a Leonard pair in $\Mat_{d+1}(\F)$.
Moreover, this Leonard pair has fundamental parameter $\beta=q+q^{-1}$
and parameter array in Proposition \ref{prop:type1closed}.
\end{proposition}

\begin{proposition}   \label{prop:type1LT}   \samepage
\ifDRAFT {\rm prop:type1LT}. \fi
Fix nonzero scalars $q, s \in \F$.
Assume the conditions {\rm (i)--(iii)} in Lemma \ref{lem:type1cond} hold.
Also assume  $s^2 \neq q^{i}$ for $0 \leq i \leq 2d-2$.
Consider the pair \eqref{eq:TDTD} with
\begin{align*}
 x_i &= s q^{1-i}                  &&  (1 \leq i \leq d),
\\
  y_i &= s^{-1} q^{i-1}            &&  (1 \leq i \leq d),
\\
 z_1 &= \frac{(q-1)(q^d-1)(s^2-q^d)}
                  {s^2 - q},
\\
 z_i &= \frac{q^{i-1}(q^i-1)(q^{d-i+1}-1)(s^2 - q^{i-2})(s^2 - q^{d+i-1})}
              {(s^2 - q^{2i-3})(s^2 - q^{2i-1})}                    && (2 \leq i \leq d-1),
\\
 z_d &= \frac{q^{d-1}(q-1)(q^d-1)(s^2 - q^{d-2})}
                  {s^2 - q^{2d-3}}.
\end{align*}
Then \eqref{eq:TDTD} is a Leonard pair in $\Mat_{d+1}(\F)$.
Moreover, this Leonard pair has fundamental parameter $\beta=q+q^{-1}$
and parameter array in Proposition \ref{prop:type1closed}.
\end{proposition}

\begin{proposition}   \label{prop:type1even}   \samepage
\ifDRAFT {\rm prop:type1even}. \fi
Fix nonzero scalars $q, s \in \F$.
Assume $d$ is even, and the conditions {\rm (i)--(iii)} in Lemma \ref{lem:type1cond} hold.
Consider the pair \eqref{eq:TDTD} with
\begin{align*}
  x_i &= s q^{1-i}            &&  (1 \leq i \leq d),
\\
  y_i &= s q^{1-i}             &&  (1 \leq i \leq d),
\\
 z_i & = 
  \begin{cases}
    q^d (q^i-1)(1- s^{-2} q^{i-2})   &  \text{ if $i$ is even},
  \\
   - q^{i-1} (q^{d-i+1}-1)(1-s^{-2} q^{d+i-1})  &  \text{ if $i$ is odd}
  \end{cases}
    & & (1 \leq i \leq d).
\end{align*}
Then \eqref{eq:TDTD} is a Leonard pair in $\Mat_{d+1}(\F)$.
Moreover, this Leonard pair has fundamental parameter $\beta=q+q^{-1}$
and parameter array in Proposition \ref{prop:type1closed}.
\end{proposition}

\section{Askey-Wilson relations}
\label{sec:AWrel}

In this section we recall the Askey-Wilson relations for a Leonard pair.
Let $A,A^*$ be a Leonard pair on $V$ with parameter array 
\[
 (\{\th_i\}_{i=0}^d, \{\th^*_i\}_{i=0}^d, \{\vphi_i\}_{i=1}^d, \{\phi_i\}_{i=1}^d)
\]
and fundamental parameter $\beta$.

\begin{lemma}  {\rm (See \cite[Theorem 11.1]{ITT}.)}    \label{lem:gammarho}   \samepage
\ifDRAFT {\rm lem:gammarho}. \fi
There exist scalars $\gamma$, $\gamma^*$, 
$\varrho$, $\varrho^*$ such that
\begin{align}
 \gamma &= \th_{i-1} - \beta \th_{i} + \th_{i+1} 
              && (1 \leq i \leq d-1),                         \label{eq:gamma}
\\
 \gamma^* &= \th^*_{i-1} - \beta \th^*_{i} + \th^*_{i+1} 
              && (1 \leq i \leq d-1),                          \label{eq:gammas}
\\
 \varrho &= \th_{i-1}^2 - \beta \th_{i-1}\th_{i} + \th_{i}^2 
             - \gamma (\th_{i-1}+\th_{i})   && ( 1 \leq i \leq d),     \label{eq:rho}
\\
 \varrho^* &= \th_{i-1}^{*2} - \beta \th^*_{i-1}\th^*_{i} + {\th_{i}^*}^2
                - \gamma^* (\th^*_{i-1}+\th^*_{i})   && (1 \leq i \leq d).     \label{eq:rhos}
\end{align}
\end{lemma}

Let the scalars $\gamma$, $\gamma^*$, $\varrho$, $\varrho^*$ be as in Lemma \ref{lem:gammarho}.

\begin{lemma} {\rm (See \cite[Theorem 1.5]{TV}.) }  \label{lem:AWrel}   \samepage
\ifDRAFT {\rm lem:AWrel}. \fi
There exist scalars $\omega$, $\eta$, $\eta^*$ such that both
\begin{align}
  A^2 A^* - \beta A A^* A + A^* A^2 - \gamma(AA^* + A^* A) - \varrho A^*
   &= \gamma^* A^2 + \omega A + \eta I,                       \label{eq:AW1orig}
\\
 {A^*}^2 A - \beta A^* A A^* + A {A^*}^2 - \gamma^* (A^*A + A A^*) - \varrho^* A
  &= \gamma {A^*}^2 + \omega A^* + \eta^* I.                \label{eq:AW2orig}
\end{align}
The scalars $\omega$, $\eta$, $\eta^*$ are uniquely determined by $A,A^*$.
\end{lemma}

The relations \eqref{eq:AW1orig} and \eqref{eq:AW2orig} are known as the {\em Askey-Wilson relations}.
Below we describe the scalars $\omega$, $\eta$, $\eta^*$.
For $0 \leq i \leq d$ let $E_i$ (resp.\ $E^*_i$) be the primitive idempotent of $A$ (resp.\ $A^*$)
associated with $\th_i$ (resp.\ $\th^*_i$).
Let the scalars $\{a_i\}_{i=0}^d$, $\{a^*_i\}_{i=0}^d$ be from Definition \ref{def:ai}.
For notational convenience,
define $\th_{-1}$, $\th_{d+1}$ (resp.\ $\th^*_{-1}$, $\th^*_{d+1}$)
so that \eqref{eq:gamma} (resp.\ \eqref{eq:gammas}) holds for $i=0$ and $i=d$.
Let the scalars $\omega$, $\eta$, $\eta^*$ be from Lemma \ref{lem:AWrel}.

\begin{lemma} {\rm (See \cite[Theorem 5.3]{TV}.) }  \label{lem:omega}              \samepage
\ifDRAFT {\rm lem:omega}. \fi
With the above notation,
\begin{align*}
  \omega &= a^*_i  (\th_i-\th_{i+1}) + a^*_{i-1}(\th_{i-1}-\th_{i-2}) - \gamma^* (\th_{i-1}+\th_i)
               && (1 \leq i \leq d),  
\\
  \eta &= a^*_i (\th_i-\th_{i-1})(\th_i-\th_{i+1}) - \gamma^* \th_i^2 - \omega \th_i
               && (0 \leq i \leq d), 
\\
 \eta^* &= a_i (\th^*_i-\th^*_{i-1})(\th^*_i -\th^*_{i+1}) - \gamma {\th^*_i}^2 - \omega \th^*_i
               && (0 \leq i \leq d). 
\end{align*}
\end{lemma}

We mention a lemma for later use.

\begin{lemma} {\rm (See \cite[Proof of Theorem 3.10]{T:tworelations}.) }   \label{lem:AWtrid}    \samepage
\ifDRAFT {\rm lem:AWtrid}. \fi
Consider linear transformations $A: V \to V$ and $A^*: V \to V$ that satisfy
\eqref{eq:AW1orig} for some scalars 
$\beta$, $\gamma$, $\gamma^*$, $\varrho$, $\omega$, $\eta \in \F$.
Assume $A$ is multiplicity-free with eigenvalues $\{\th_i\}_{i=0}^d$.
For $0 \leq i \leq d$ let $E_i$ be the primitive idempotent of $A$ associated with $\th_i$.
Assume that for $0 \leq i,j \leq d$
\begin{align*}
 \th_i^2 - \beta \th_i \th_j + \th_j^2 - \gamma (\th_i + \th_j) - \varrho &\neq 0 
   &&  \text{ if $\;|i-j|>1$.} 
\end{align*}
Then $E_i A^* E_j = 0$ if $\;|i-j|>1$ for $0 \leq i,j \leq d$.
\end{lemma}

Below we obtain the scalars $\gamma$, $\gamma^*$, $\varrho$, $\varrho^*$, $\omega$,
$\eta$, $\eta^*$  for a Leonard pair that is isomorphic to its opposite.

\begin{lemma}    \label{lem:AWtype2}   \samepage
\ifDRAFT {\rm lem:AWtype2}. \fi
Assume $\beta = 2$, and the parameter array satisfies
\eqref{eq:type2th}--\eqref{eq:type2phi} for a nonzero $s \in \F$.
Then $\gamma=0$, $\gamma^*=0$, $\eta=0$, $\eta^*=0$, and
\begin{align*}
 \varrho &= 4, & \varrho^* &= 4,
& \omega &= -2 (s+s^{-1}).
\end{align*}
\end{lemma}

\begin{lemma}    \label{lem:AWtype3}   \samepage
\ifDRAFT {\rm lem:AWtype3}. \fi
Assume $\beta = -2$, and the parameter array satisfies
\eqref{eq:type3th}--\eqref{eq:type3phi} for a scalar  $\tau \in \F$.
Then 
 $\gamma=0$, $\gamma^*=0$,
$\eta=0$, $\eta^*=0$, and
\begin{align*}
  \varrho &= 4, & \varrho^* &= 4,
& \omega &= 4(d+1)\tau.
\end{align*}
\end{lemma}

\begin{lemma}    \label{lem:AWtype1}   \samepage
\ifDRAFT {\rm lem:AWtype1}. \fi
Assume $\beta \neq 2$, $\beta \neq -2$, and pick a nonzero $q \in \F$
such that $\beta = q+q^{-1}$.
Assume the parameter array satisfies
\eqref{eq:type1th}--\eqref{eq:type1phi} for a nonzero $s \in \F$.
Then $\gamma = 0$, $\gamma^* = 0$, $\eta = 0$, $\eta^* = 0$, and
\begin{align*}
\varrho &= q^{d-2}(q^2-1)^2,  \qquad\qquad  \varrho^* = q^{d-2}(q^2-1)^2,
\\
\omega &= - q^{-1} (q-1)^2 (q^{d+1}+1) (s+s^{-1}q^{d-1}).
\end{align*}
\end{lemma}

\section{The characteristic polynomial of a zero-diagonal TD matrix}
\label{sec:char}

In this section we display a formula for the characteristic polynomial of 
a zero-diagonal TD matrix. 
Let $A \in \Mat_{d+1}(\F)$ be a zero-diagonal TD matrix.
In view of Note \ref{note:subdiagonal1}, there exists an invertible diagonal matrix
$D \in \Mat_{d+1}(\F)$ such that $D^{-1} A D$ has all subdiagonal entries $1$.
Clearly $A$ and $D^{-1}AD$ has the same characteristic polynomial.
So we assume
\[
 A =
 \begin{pmatrix}
  0  & z_1 & & & & \text{\bf 0} \\
  1 & 0  & z_2 \\
   & 1 & 0 & \cdot \\
   & & \cdot & \cdot & \cdot \\
   & & & \cdot & \cdot & z_d \\
  \text{\bf 0}  & & & & 1 & 0
 \end{pmatrix}.
\]

\begin{definition}    \label{def:pi}   \samepage
\ifDRAFT {\rm def:pi}. \fi
For an integer $i$ with $0 \leq i \leq \lfloor (d+1)/2 \rfloor$, let $\pi (d,i)$ denote the sum of $z_{\ell_1} z_{\ell_2} \cdots z_{\ell_i}$
over all $(\ell_1, \ell_2, \ldots, \ell_i)$ such that 
$1 \leq \ell_1, \ell_2, \ldots, \ell_i \leq d$ and $\ell_{j+1} - \ell_j \geq 2$ for $1 \leq j \leq i-1$.
\end{definition}

Let $f(x)$ be the characteristic polynomial of $A$:
\[
   f(x) = \text{det} (x I - A).
\]
Then
\begin{equation}         \label{eq:f(x)}
 f(x) = \sum_{i=0}^{\lfloor (d+1)/2 \rfloor} (-1)^i \pi (d,i)\, x^{d-2i+1}.
\end{equation}
The proof of \eqref{eq:f(x)} is routine using induction on $d$.

\begin{example}
When $d=5$,
\begin{align*}
 f(x) &= x^6 - x^4 ( z_1 + z_2 + z_3 + z_4 + z_5 )
\\  & \qquad\;
           + x^2 (z_1 z_3 + z_1 z_4 + z_1 z_5 + z_2 z_4 + z_2 z_5 + z_3 z_5)
           -  z_1 z_3 z_5.
\end{align*}
When $d=6$,
\begin{align*}
 f(x) &= x^7 - x^5 (  z_1 + z_2 + z_3 + z_4 + z_5 + z_6)
\\ & \qquad\;
      + x^3 ( z_1 z_3 + z_1 z_4 + z_1 z_5 + z_1 z_6 + z_2 z_4 + z_2 z_5 + z_2 z_6 + z_3 z_5 + z_3 z_6 + z_4 z_6)
\\ & \qquad\;
      - x (z_1 z_3 z_5 + z_1 z_3 z_6 + z_1 z_4 z_6 + z_2 z_4 z_6).
\end{align*}
\end{example}

\section{Proof of Proposition \ref{prop:type2ex} }
\label{sec:type2ex}

Fix a nonzero $s \in \F$, and assume conditions (i), (ii) in Lemma \ref{lem:type2cond} hold.
Define scalars $\{x_i\}_{i=1}^d$, $\{y_i\}_{i=1}^d$, $\{z_i\}_{i=1}^d$ as in Proposition \ref{prop:type2ex},
and let $A,A^*$ be the zero-diagonal TD-TD pair \eqref{eq:TDTD}.
Define scalars $\{\th_i\}_{i=0}^d$, $\{\th^*_i\}_{i=0}^d$ by \eqref{eq:type2th}, \eqref{eq:type2ths}.

\begin{lemma}    \label{lem:type2thths}   \samepage
\ifDRAFT {\rm lem:type2thths}. \fi
The scalars $\{\th_i\}_{i=0}^d$ (resp.\ $\{\th^*_i\}_{i=0}^d$) are mutually distinct.
Moreover $A$ (resp.\ $A^*$) has eigenvalues $\{\th_i\}_{i=0}^d$ (resp.\ $\{\th^*_i\}_{i=0}^d$).
\end{lemma}

\begin{proof}
By Lemma \ref{lem:type2cond}(i) the scalars $\{\th_i\}_{i=0}^d$ are mutually distinct.
Using \eqref{eq:f(x)} one 
checks that $\text{det}(\th_i I - A) = 0$ for $0 \leq i \leq d$.
So $\th_i$ is a root of the characteristic polynomial of $A$.
Therefore $\{\th_i\}_{i=0}^d$ are the eigenvalues of $A$.
The proof for $A^*$ is similar.
\end{proof}

Define $A^\ve \in \Mat_{d+1}(\F)$ by 
\begin{equation}
  A^\ve = A A^* - A^* A.  \label{eq:type2defAe}
\end{equation}
Define scalars $\{\th^\ve_i\}_{i=0}^d$ by
\begin{align*}
      \th^{\ve}_i &= (d-2i)(s-s^{-1})     &&  (0 \leq i \leq d).
\end{align*}

\begin{lemma}    \label{lem:type2Aemultfree}   \samepage
\ifDRAFT {\rm lem:type2Aemultfree}. \fi
The scalars $\{\th^\ve_i\}_{i=0}^d$ are mutually distinct.
Moreover 
\begin{equation}             \label{eq:type2Aediag}
   A^\ve = \text{\rm diag} (\th^\ve_0, \th^\ve_1, \ldots, \th^\ve_d).
\end{equation}
\end{lemma}

\begin{proof}
The scalars $\{\th^\ve_i\}_{i=0}^d$ are mutually distinct
since $s - s^{-1} \neq 0$ by Lemma \ref{lem:type2cond}(ii).
One routinely checks \eqref{eq:type2Aediag}.
\end{proof}

\begin{lemma}    \label{lem:AWtype2short}   \samepage
\ifDRAFT {\rm lem:AWtype2short}. \fi
The matrices $A$, $A^*$, $A^\ve$ satisfy
\begin{align}
  A^* A^\ve - A^\ve A^* &= - 4A + 2(s+s^{-1}) A^*,    \label{eq:type2AW1short}
\\
 A^\ve A - A A^\ve &= - 4 A^* + 2(s+s^{-1}) A.          \label{eq:type2AW2short}
\end{align}
\end{lemma}

\begin{proof}
Routine verification.
\end{proof}

Let the scalars $\beta$,  $\gamma$, $\gamma^*$, $\varrho$, $\varrho^*$, $\omega$, $\eta$, $\eta^*$
be as in Lemma \ref{lem:AWtype2}.

\begin{lemma}   \label{lem:type2AWrel}   \samepage
\ifDRAFT {\rm lem:type2AWrel}. \fi
The matrices $A$, $A^*$ satisfy \eqref{eq:AW1orig} and \eqref{eq:AW2orig}.
\end{lemma}

\begin{proof}
In \eqref{eq:type2AW1short} and \eqref{eq:type2AW2short},
eliminate $A^\ve$ using \eqref{eq:type2defAe}.
\end{proof}

For $0 \leq i \leq d$ let $E_i$ (resp.\ $E^*_i$) be the primitive idempotent of $A$
(resp.\ $A^*$) associated with $\th_i$ (resp.\ $\th^*_i$).

\begin{lemma}    \label{lem:type2EiAsEj}    \samepage
\ifDRAFT {\rm lem :type2EiAsEj}.  \fi
For $0 \leq i,j \leq d$ such that $|i-j| > 1$,
\begin{align*}
  E_i A^* E_j &= 0,   &  E^*_i A E^*_j &= 0.
\end{align*}
\end{lemma}

\begin{proof}
We have $\gamma=0$ and $\varrho=4$.
By this and \eqref{eq:type2th},
\[
  \th_i^2 - \beta \th_i \th_j + \th_j^2 - \gamma (\th_i+\th_j) - \varrho
  = 4 (i-j-1)(i-j+1) \neq 0.
\]
Now $E_i A^* E_j = 0$ by Lemma \ref{lem:AWtrid}.
The proof of $E^*_i A E^*_j = 0$ is similar.
\end{proof}

For $0 \leq i \leq d$ let $E^\ve_i$ be the primitive idempotent of $A^\ve$
associate with $\th^\ve_i$.
Consider the sequence 
\[
  \Phi^\ve = (A, \{E_i\}_{i=0}^d, A^\ve, \{E^\ve_i\}_{i=0}^d).
\]

\begin{lemma}    \label{lem:type2Phie}   \samepage
\ifDRAFT {\rm lem:type2Phie}. \fi
$\Phi^\ve$ is a Leonard system.
\end{lemma}

\begin{proof}
We verify conditions (i)--(v) in Definition \ref{def:LS}.
By Lemmas \ref{lem:type2thths} and \ref{lem:type2Aemultfree}
each of $A$, $A^\ve$ is multiplicity-free, so condition (i) holds.
By the construction, conditions (ii), (iii) hold.
Concerning condition (iv),
pick integers $i$, $j$ such that $0 \leq i,j \leq d$ and $|i-j|>1$.
By the shape of $A$ we have $E^\ve_i A E^\ve_j = 0$.
We show $E_i A^\ve E_j = 0$.
In \eqref{eq:type2AW2short}, multiply each side on the left by $E_i$
and on the right by $E_j$ to find
\[
   (\th_j - \th_i) E_i A^\ve E_j = - 4 E_i A^* E_j.
\]
By Lemma \ref{lem:type2EiAsEj} $E_i A^* E_j=0$.
By these comments $E_i A^\ve E_j = 0$.
Thus condition (iv) holds.
Concerning condition (v), pick integers $i$, $j$ such that
$0 \leq i,j \leq d$ and $|i-j|=1$.
We have $E^\ve_i A E^\ve_j \neq 0$ by the shape of $A$.
Now apply Lemma \ref{lem:irred} to $\Phi^\ve$ to find that
$\Phi^\ve$ is a Leonard system.
\end{proof}

\begin{lemma}    \label{lem:type2generate}   \samepage
\ifDRAFT {\rm lem:type2generate}. \fi
$A$ and  $A^*$ together generate $\Mat_{d+1}(\F)$.
\end{lemma}

\begin{proof}
By Lemmas \ref{lem:irred} and \ref{lem:type2Phie} 
 $A^\ve$ and $A$ together generate $\Mat_{d+1}(\F)$.
By \eqref{eq:type2defAe} $A^\ve$ is a polynomial in $A$ and $A^*$.
By these comments $A$ and $A^*$ together generate $\Mat_{d+1}(\F)$.
\end{proof}

\begin{proofof}{Proposition \ref{prop:type2ex}}
Consider the sequence $\Phi = (A, \{E_i\}_{i=0}^d, A^*, \{E^*_i\}_{i=0}^d)$.
We check  conditions (i)--(v) in Definition \ref{def:LS}.
By Lemma \ref{lem:type2thths} each of $A,A^*$ is multiplicity-free, so condition (i) holds.
By the construction conditions (ii) and (iii) holds.
By Lemma \ref{lem:type2EiAsEj} condition (iv) holds.
By Lemmas \ref{lem:irred} and \ref{lem:type2generate} condition (v) holds.
Thus $\Phi$ is a Leonard system, and so $A,A^*$ is a Leonard pair.
Concerning the parameter array of $A,A^*$,
define $\{\vphi_i\}_{i=1}^d$ and $\{\phi_i\}_{i=1}^d$ by \eqref{eq:vphiformula} and \eqref{eq:phiformula}.
One routinely checks that 
\[
    (\{\th_i\}_{i=0}^d, \{\th^*_i\}_{i=0}^d, \{\vphi_i\}_{i=1}^d, \{\phi_i\}_{i=1}^d)
\]
coincides with the parameter array in Proposition \ref{prop:type2closed}.
\end{proofof}

\section{Proof of Proposition \ref{prop:type3ex}}
\label{sec:type3ex}

Fix  $\tau \in \F$, and assume conditions (i), (ii) in Lemma \ref{lem:type3cond} hold.
Note that $d$ is even and $\text{Char}(\F) \neq 2$.
Fix $\epsilon \in \{1, -1\}$,
and define scalars $\{x_i\}_{i=1}^d$, $\{y_i\}_{i=1}^d$, $\{z_i\}_{i=1}^d$ as in Proposition \ref{prop:type3ex}.
Let $A,A^*$ be the zero-diagonal TD-TD pair \eqref{eq:TDTD}.
Define scalars $\{\th_i\}_{i=0}^d$, $\{\th^*_i\}_{i=0}^d$ by \eqref{eq:type3th}, \eqref{eq:type3ths}.

\begin{lemma}    \label{lem:type3thths}   \samepage
\ifDRAFT {\rm lem:typeethths}. \fi
The scalars $\{\th_i\}_{i=0}^d$ (resp.\ $\{\th^*_i\}_{i=0}^d$) are mutually distinct.
Moreover $A$ (resp.\ $A^*$) has eigenvalues $\{\th_i\}_{i=0}^d$ (resp.\ $\{\th^*_i\}_{i=0}^d$).
\end{lemma}

\begin{proof}
By Lemma \ref{lem:type3cond}(i)  the scalars $\{\th_i\}_{i=0}^d$ are mutually distinct.
Using \eqref{eq:f(x)}
one checks that $\text{det}(\th_i I - A) = 0$ for $0 \leq i \leq d$.
So $\th_i$ is a root of the characteristic polynomial of $A$.
Therefore $\{\th_i\}_{i=0}^d$ are the eigenvalues of $A$.
The proof for $A^*$ is similar.
\end{proof}

Define $A^\ve \in \Mat_{d+1}(\F)$ by
\begin{equation}
    A^\ve = A A^* + A^* A.                      \label{eq:type3defAe}
\end{equation}
Define scalars $\{\th^\ve_i\}_{i=0}^d$ by
\begin{align*}
  \th^{\ve}_i =
    \begin{cases}
       2 d \tau + 2(d-2i) \epsilon   &  \text{ if $i$ is even},
    \\
      (2d+4)\tau - 2(d-2i) \epsilon  &  \text{ if $i$ is odd}
    \end{cases}
      &&  (0 \leq i \leq d).  
\end{align*}

\begin{lemma}    \label{lem:type3Aemultfree}   \samepage
\ifDRAFT {\rm lem:type3Aemultfree}. \fi
The scalars $\{\th^\ve_i\}_{i=0}^d$ are mutually distinct.
Moreover
\begin{equation}             \label{eq:type3Aediag}
   A^\ve = \text{\rm diag} (\th^\ve_0, \th^\ve_1, \ldots, \th^\ve_d).
\end{equation}
\end{lemma}

\begin{proof}
The scalars $\{\th^\ve_i\}_{i=0}^d$ are mutually distinct
by conditions (i), (ii) in Lemma \ref{lem:type3cond}.
One routinely checks \eqref{eq:type2Aediag}.
\end{proof}

\begin{lemma}    \label{lem:AWtype3short}   \samepage
\ifDRAFT {\rm lem:AWtype3short}. \fi
The matrices $A,$ $A^*$, $A^\ve$ satisfy
\begin{align}
  A^* A^\ve + A^\ve A^* &=  4A + 4(d+1) \tau A^*,    \label{eq:type3AW1short}
\\
 A^\ve A + A A^\ve &=  4 A^* + 4(d+1) \tau  A.          \label{eq:type3AW2short}
\end{align}
\end{lemma}

\begin{proof}
Routine verification.
\end{proof}

Let the scalars $\beta$, $\gamma$, $\gamma^*$, $\varrho$, $\varrho^*$, $\omega$, $\eta$, $\eta^*$
be as in Lemma \ref{lem:AWtype3}.

\begin{lemma}   \label{lem:type3AWrel}   \samepage
\ifDRAFT {\rm lem:type3AWrel}. \fi
The matrices $A$, $A^*$ satisfy \eqref{eq:AW1orig} and \eqref{eq:AW2orig}.
\end{lemma}

\begin{proof}
In \eqref{eq:type3AW1short} and \eqref{eq:type3AW2short}, eliminate $A^\ve$ using \eqref{eq:type3defAe}.
\end{proof}

For $0 \leq i \leq d$ let $E_i$ (resp.\ $E^*_i$) be the primitive idempotent of $A$
(resp.\ $A^*$) associated with $\th_i$ (resp.\ $\th^*_i$).

\begin{lemma}    \label{lem:type3EiAsEj}    \samepage
\ifDRAFT {\rm lem :type3EiAsEj}.  \fi
For $0 \leq i,j \leq d$ such that $|i-j| > 1$,
\begin{align*}
  E_i A^* E_j &= 0,   &  E^*_i A E^*_j &= 0.
\end{align*}
\end{lemma}

\begin{proof}
We have $\gamma = 0$ and $\varrho = 4$,
so
\[
  \th_i^2 - \beta \th_i \th_j + \th_j^2 - \gamma (\th_i+\th_j) - \varrho
  = (\th_i + \th_j -2)(\th_i + \th_j + 2).
\]
Using \eqref{eq:type3th}
\begin{align*}
  \th_i + \th_j &=
   \begin{cases}
     2(i+j-d)   &  \text{ if $i$ is even, $j$ is even},
  \\
    2 (i-j)      &  \text{ if $i$ is even, $j$ is odd},
  \\
    2(j-i)       &  \text{ if $i$ is odd, $j$ is even},
  \\
   2(d-i-j)     &  \text{ if $i$ is odd, $j$ is odd}
  \end{cases}
  &&  (0 \leq i,j \leq d).
\end{align*}
Using this and condition (i) in Lemma \ref{lem:type3cond}, one checks
$\th_i+\th_j-2 \neq 0$ and $\th_i+\th_j+2 \neq 0$ if $|i-j| > 1$.
By this and Lemma \ref{lem:AWtrid} 
$E_i A^* E_j = 0$.
The proof of $E^*_i A E^*_j =0$ is similar.
\end{proof}

For $0 \leq i \leq d$ let $E^\ve_i$ be the primitive idempotent of $A^\ve$
associate with $\th^\ve_i$.
Consider the sequence 
\[
  \Phi^\ve = (A, \{E_i\}_{i=0}^d, A^\ve, \{E^\ve_i\}_{i=0}^d).
\]

\begin{lemma}    \label{lem:type3Phie}   \samepage
\ifDRAFT {\rm lem:type3Phie}. \fi
$\Phi^\ve$ is a Leonard system.
\end{lemma}

\begin{proof}
Similar to the proof of Lemma \ref{lem:type2Phie}.
\end{proof}

\begin{lemma}    \label{lem:type3generate}   \samepage
\ifDRAFT {\rm lem:type3generate}. \fi
The matrices $A$ and $A^*$ together generate $\Mat_{d+1}(\F)$.
\end{lemma}

\begin{proof}
Similar to the proof of Lemma \ref{lem:type2generate}.
\end{proof}

\begin{proofof}{Proposition \ref{prop:type3ex}}
Consider the sequence $\Phi = (A, \{E_i\}_{i=0}^d, A^*, \{E^*_i\}_{i=0}^d)$.
We check conditions (i)--(v) in Definition \ref{def:LS}.
By Lemma \ref{lem:type3thths} each of $A,A^*$ is multiplicity-free, so condition (i) holds.
By the construction conditions (ii) and (iii) holds.
By Lemma \ref{lem:type3EiAsEj} condition (iv) holds.
By Lemmas \ref{lem:irred} and \ref{lem:type3generate} conditions (v) holds.
Thus $\Phi$ is a Leonard system, and so $A,A^*$ is a Leonard pair.
One can show that $A,A^*$ has parameter array in Proposition \ref{prop:type3closed}
in a similar way as in the proof of Proposition \ref{prop:type2ex}.
\end{proofof}

\section{Proof of Proposition \ref{prop:type1LT}}
\label{sec:type1LT}

Fix a nonzero $q$, $s \in \F$,
and assume conditions (i)--(iii) in Lemma \ref{lem:type1cond} hold.
Also assume 
\begin{align}
   s^2 &\neq q^i   &&  (0 \leq i \leq 2d-2).            \label{eq:type1LTconds}
\end{align}
Define scalars $\{x_i\}_{i=1}^d$, $\{y_i\}_{i=1}^d$, $\{z_i\}_{i=1}^d$ as in Proposition \ref{prop:type1LT},
and let $A,A^*$ be the zero-diagonal TD-TD pair \eqref{eq:TDTD}.
Define scalars $\{\th_i\}_{i=0}^d$, $\{\th^*_i\}_{i=0}^d$ by \eqref{eq:type1th}, \eqref{eq:type1ths}.

\begin{lemma}    \label{lem:type1LTthths}   \samepage
\ifDRAFT {\rm lem:type1LTthths}. \fi
The scalars $\{\th_i\}_{i=0}^d$ (resp.\ $\{\th^*_i\}_{i=0}^d$) are mutually distinct.
Moreover $A$ (resp.\ $A^*$) has eigenvalues $\{\th_i\}_{i=0}^d$ (resp.\ $\{\th^*_i\}_{i=0}^d$).
\end{lemma}

\begin{proof}
By conditions (i), (ii) in Lemma \ref{lem:type1cond} the scalars $\{\th_i\}_{i=0}^d$ are mutually distinct.
Using \eqref{eq:f(x)} 
one checks that $\text{det}(\th_i I - A) = 0$ for $0 \leq i \leq d$.
So $\th_i$ is a root of the characteristic polynomial of $A$.
Therefore $\{\th_i\}_{i=0}^d$ are the eigenvalues of $A$.
The proof for $A^*$ is similar.
\end{proof}

Define $A^\ve \in \Mat_{d+1}(\F)$ by
\begin{equation}
    A^\ve =  A A^* - q A^* A.                            \label{eq:type1defAe}
\end{equation}
Define scalars $\{\th^\ve_i\}_{i=0}^d$ by
\begin{align*}
 \th^\ve_i &=q^{d-i}(q^2-1)(s+s^{-1} q^{2i-1}) - (q-1)(q^{d+1}+1)(s+s^{-1} q^{d-1})
           && (0 \leq i \leq d).  
\end{align*}

\begin{lemma}    \label{lem:type1Aemultfree}   \samepage
\ifDRAFT {\rm lem:type1Aemultfree}. \fi
The scalars $\{\th^\ve_i\}_{i=0}^d$ are mutually distinct.
Moreover
\begin{equation}             \label{eq:type1Aediag}
   A^\ve = \text{\rm diag} (\th^\ve_0, \th^\ve_1, \ldots, \th^\ve_d).
\end{equation}
\end{lemma}

\begin{proof}
For $0 \leq i < j \leq d$
\[
  \th^\ve_i - \th^\ve_j
 = s^{-1} q^{d-j} (q^2-1)(q^{j-i}-1)(s^2 - q^{i+j-1}).
\]
In this line, the right-hand side is nonzero by Lemma \ref{lem:type1cond}(i)
and \eqref{eq:type1LTconds}.
So $\{\th^\ve_i\}_{i=0}^d$ are mutually distinct.
One routinely checks \eqref{eq:type1Aediag}.
\end{proof}

Let the scalars $\beta$, $\gamma$, $\gamma^*$, $\varrho$, $\varrho^*$, $\omega$, $\eta$, $\eta^*$
be as in Lemma \ref{lem:AWtype1}.

\begin{lemma}    \label{lem:AWtype1short}   \samepage
\ifDRAFT {\rm lem:AWtype1short}. \fi
The matrices $A$, $A^*$, $A^\ve$ satisfy
\begin{align}
   A^* A^\ve - q A^\ve A^* &=  -q \varrho^* A - q \omega A^*,    \label{eq:type1AW1short}
\\
 A^\ve A - q A A^\ve &=  - q \varrho A^* - q \omega A.          \label{eq:type1AW2short}
\end{align}
\end{lemma}

\begin{proof}
Routine verification.
\end{proof}

\begin{lemma}   \label{lem:type1LTAWrel}   \samepage
\ifDRAFT {\rm lem:type1LTAWrel}. \fi
The matrices $A$, $A^*$ satisfy \eqref{eq:AW1orig} and \eqref{eq:AW2orig}.
\end{lemma}

\begin{proof}
In \eqref{eq:type1AW1short} and \eqref{eq:type1AW2short},
eliminate $A^\ve$ using \eqref{eq:type1defAe}.
\end{proof}

For $0 \leq i \leq d$ let $E_i$ (resp.\ $E^*_i$) be the primitive idempotent of $A$
(resp.\ $A^*$) associated with $\th_i$ (resp.\ $\th^*_i$).

\begin{lemma}    \label{lem:type1EiAsEj}    \samepage
\ifDRAFT {\rm lem :type1EiAsEj}.  \fi
For $0 \leq i,j \leq d$ such that $|i-j| > 1$,
\begin{align*}
  E_i A^* E_j &= 0,   &  E^*_i A E^*_j &= 0.
\end{align*}
\end{lemma}

\begin{proof}
We have $\gamma = 0$ and $\varrho = q^{d-2}(q^2-1)^2$.
By this and \eqref{eq:type1th},
\begin{align*}
 &  \th_i^2 - \beta \th_i \th_j + \th_j^2 - \gamma (\th_i+\th_j) - \varrho
\\
& \qquad\qquad  = q^{2j} (q^{i-j+1}-1)(q^{i-j-1}-1)(q^{d-i-j-1}+1)(q^{d-i-j+1}+1).
\end{align*}
In this equation, the right-hand side is nonzero by conditions (i), (ii) in Lemma \ref{lem:type1cond}.
So $E_i A^* E_j = 0$ by Lemma \ref{lem:AWtrid}.
The proof of $E^*_i A E^*_j =0$ is similar.
\end{proof}

For $0 \leq i \leq d$ let $E^\ve_i$ be the primitive idempotent of $A^\ve$
associate with $\th^\ve_i$.
Consider the sequence 
\[
  \Phi^\ve = (A, \{E_i\}_{i=0}^d, A^\ve, \{E^\ve_i\}_{i=0}^d).
\]

\begin{lemma}    \label{lem:type1Phie}   \samepage
\ifDRAFT {\rm lem:type1Phie}. \fi
$\Phi^\ve$ is a Leonard system.
\end{lemma}

\begin{proof}
Similar to the proof of Lemma \ref{lem:type2Phie}.
\end{proof}

\begin{lemma}    \label{lem:type1generate}   \samepage
\ifDRAFT {\rm lem:type1generate}. \fi
The matrices $A$ and $A^*$ together generate $\Mat_{d+1}(\F)$.
\end{lemma}

\begin{proof}
Similar to the proof of Lemma \ref{lem:type2generate}.
\end{proof}

\begin{proofof}{Proposition \ref{prop:type1LT}}
Consider the sequence $\Phi = (A, \{E_i\}_{i=0}^d, A^*, \{E^*_i\}_{i=0}^d)$.
We check conditions (i)--(v) in Definition \ref{def:LS}.
By Lemma \ref{lem:type1LTthths} each of $A,A^*$ is multiplicity-free,
so condition (i) holds.
By the construction conditions (ii) and (iii) hold.
By Lemma \ref{lem:type1EiAsEj} condition (iv) holds.
By Lemmas \ref{lem:irred} and \ref{lem:type1generate} condition (v) holds.
Thus $\Phi$ is a Leonard system, and so $A,A^*$ is a Leonard pair.
One can show that $A,A^*$ has parameter array in Proposition \ref{prop:type1closed}
in a similar way as in the proof of Proposition \ref{prop:type2ex}.
\end{proofof}

\section{Proof of Proposition \ref{prop:type1compact}}
\label{sec:type1compact}

Fix a nonzero $q$, $s \in \F$,
and assume conditions (i)--(iii) in Lemma \ref{lem:type1cond} hold.
Define scalars $\{x_i\}_{i=1}^d$, $\{y_i\}_{i=1}^d$, $\{z_i\}_{i=1}^d$ as in Proposition \ref{prop:type1compact},
and let $A,A^*$ be the zero-diagonal TD-TD pair \eqref{eq:TDTD}.
Define scalars $\{\th_i\}_{i=0}^d$, $\{\th^*_i\}_{i=0}^d$ by \eqref{eq:type1th}, \eqref{eq:type1ths}.

\begin{lemma}    \label{lem:type1compactthths}   \samepage
\ifDRAFT {\rm lem:type1compactthths}. \fi
The scalars $\{\th_i\}_{i=0}^d$ (resp.\ $\{\th^*_i\}_{i=0}^d$) are mutually distinct.
Moreover $A$ (resp.\ $A^*$) has eigenvalues $\{\th_i\}_{i=0}^d$ (resp.\ $\{\th^*_i\}_{i=0}^d$).
\end{lemma}

\begin{proof}
Similar to the proof of Lemma \ref{lem:type1LTthths}.
\end{proof}

Let the scalars $\beta$, $\gamma$, $\gamma^*$, $\varrho$, $\varrho^*$, $\omega$, $\eta$, $\eta^*$
be as in Lemma \ref{lem:AWtype1}.

\begin{lemma}   \label{lem:type1compactAWrel}   \samepage
\ifDRAFT {\rm lem:type1compactAWrel}. \fi
The matrices $A$, $A^*$ satisfy \eqref{eq:AW1orig} and \eqref{eq:AW2orig}.
\end{lemma}

\begin{proof}
Routine verification.
\end{proof}

For $0 \leq i \leq d$ let $E_i$ (resp.\ $E^*_i$) be the primitive idempotent of $A$
(resp.\ $A^*$) associated with $\th_i$ (resp.\ $\th^*_i$).

\begin{lemma}    \label{lem:type1compactEiAsEj}    \samepage
\ifDRAFT {\rm lem :type1compactEiAsEj}.  \fi
For $0 \leq i,j \leq d$ such that $|i-j| > 1$,
\begin{align*}
  E_i A^* E_j &= 0,   &  E^*_i A E^*_j &= 0.
\end{align*}
\end{lemma}

\begin{proof}
Similar to the proof of Lemma \ref{lem:type1EiAsEj}.
\end{proof}

For $a \in \F$ and an integer $n \geq 0$, define
\[
  (a;q)_n = (1-a)(1-a q)(1-a q^2) \cdots (1-a q^{n-1}).
\]
We interpret $(a;q)_0=1$.

\begin{lemma}    \label{lem:type1compactEr-1AsEr}    \samepage
\ifDRAFT {\rm lem:type1compactEr-1AsEr}. \fi
For $1 \leq r \leq d$
\begin{align*}
  E_{r-1} A^* E_r & \neq 0,  &  E_r A^* E_{r-1} &\neq 0.
\end{align*}
\end{lemma}

\begin{proof}
One routinely checks that for $1 \leq r \leq d$,
\begin{align*}
 (E_{r-1} A^* E_r)_{0,0} & =
\frac{ s q^{d-2r+1} (1-s^{-2}q^{2r-2}) (q:q)_d }
       {2 (1+q^{d-2r+1}) (q^2;q^2)_{r-1} (q^2;q^2)_{d-r}},
\\
 (E_{r} A^* E_{r-1})_{0,0} & =
 - \frac{ s (1-s^{-2}q^{2d-2r}) (q:q)_d }
           {2 (1+q^{d-2r+1}) (q^2;q^2)_{r-1} (q^2;q^2)_{d-r}}.
\end{align*}
By this and using conditions (i)--(iii) in Lemma \ref{lem:type1cond} one finds
$(E_{r-1}A^*E_r)_{0,0} \neq 0$ and $(E_r A^* E_{r-1})_{0,0} \neq 0$.
The result follows.
\end{proof}

\begin{proofof}{Proposition \ref{prop:type1compact}}
Consider the sequence $\Phi = (A, \{E_i\}_{i=0}^d, A^*, \{E^*_i\}_{i=0}^d)$.
We check conditions (i)--(v) in Definition \ref{def:LS}.
By Lemma \ref{lem:type1compactthths} each of $A,A^*$ is multiplicity-free, so condition (i) holds.
By the construction conditions (ii) and (iii) hold.
By Lemmas \ref{lem:type1compactEiAsEj} conditions (iv) holds.
By Lemmas \ref{lem:irred} and \ref{lem:type1compactEr-1AsEr} condition (v) holds.
Thus $\Phi$ is a Leonard system, and so $A,A^*$ is a Leonard pair.
One can show that $A,A^*$ has parameter array in Proposition \ref{prop:type1closed}
in a similar way as in the proof of Proposition \ref{prop:type2ex}.
\end{proofof}

\section{Proof of Proposition \ref{prop:type1even}}
\label{sec:type1even}

Assume $d$ is even.
Fix a nonzero $q$, $s \in \F$,
and assume conditions (i)--(iii) in Lemma \ref{lem:type1cond} hold.
Define scalars $\{x_i\}_{i=1}^d$, $\{y_i\}_{i=1}^d$, $\{z_i\}_{i=1}^d$ as in Proposition \ref{prop:type1even},
and let $A,A^*$ be the zero-diagonal TD-TD pair \eqref{eq:TDTD}.
Define scalars $\{\th_i\}_{i=0}^d$, $\{\th^*_i\}_{i=0}^d$ by \eqref{eq:type1th}, \eqref{eq:type1ths}.

\begin{lemma}    \label{lem:type1eventhths}   \samepage
\ifDRAFT {\rm lem:type1eventhths}. \fi
The scalars $\{\th_i\}_{i=0}^d$ (resp.\ $\{\th^*_i\}_{i=0}^d$) are mutually distinct.
Moreover $A$ (resp.\ $A^*$) has eigenvalues $\{\th_i\}_{i=0}^d$ (resp.\ $\{\th^*_i\}_{i=0}^d$).
\end{lemma}

\begin{proof}
Similar to the proof of Lemma \ref{lem:type1LTthths}.
\end{proof}

Let the scalars $\beta$, $\gamma$, $\gamma^*$, $\varrho$, $\varrho^*$, $\omega$, $\eta$, $\eta^*$
be as in Lemma \ref{lem:AWtype1}.

\begin{lemma}   \label{lem:type1evenAWrel}   \samepage
\ifDRAFT {\rm lem:type1evenAWrel}. \fi
The matrices $A$, $A^*$ satisfy \eqref{eq:AW1orig} and \eqref{eq:AW2orig}.
\end{lemma}

\begin{proof}
Routine verification.
\end{proof}

For $0 \leq i \leq d$ let $E_i$ (resp.\ $E^*_i$) be the primitive idempotent of $A$
(resp.\ $A^*$) associated with $\th_i$ (resp.\ $\th^*_i$).

\begin{lemma}    \label{lem:type1evenEiAsEj}    \samepage
\ifDRAFT {\rm lem :type1evenEiAsEj}. \fi
For $0 \leq i,j \leq d$ such that $|i-j| > 1$,
\begin{align*}
  E_i A^* E_j &= 0,   &  E^*_i A E^*_j &= 0.
\end{align*}
\end{lemma}

\begin{proof}
Similar to the proof of Lemma \ref{lem:type1EiAsEj}.
\end{proof}

\begin{lemma}    \label{lem:type1evenEr-1AsEr}    \samepage
\ifDRAFT {\rm lem:type1evenEr-1AsEr}. \fi
For $1 \leq r \leq d$
\begin{align*}
  E_{r-1} A^* E_r & \neq 0,  &  E_r A^* E_{r-1} &\neq 0.
\end{align*}
\end{lemma}

\begin{proof}
One routinely checks that for $1 \leq r \leq d$,
\begin{align*}
(E_{r-1} A^* E_r)_{0,0}  & =
  \frac{(-1)^{r-1} s q^{r(d-r)} (q^2;q^2)_{d/2} (s^{-2};q^2)_r (s^{-2};q^2)_{d-r+1} }
         {2 (1 + q^{d-2r+1}) (q^2;q^2)_{r-1} (q^2;q^2)_{d-r} (s^{-2}; q^2)_{d/2}},
\\
(E_{r} A^* E_{r-1})_{0,0}  & =
  \frac{(-1)^{r-1} s q^{r(d-r)} (q^2;q^2)_{d/2} (s^{-2};q^2)_r (s^{-2};q^2)_{d-r+1} }
         {2 (1 + q^{d-2r+1}) (q^2;q^2)_{r-1} (q^2;q^2)_{d-r} (s^{-2}; q^2)_{d/2}}.
\end{align*}
By this and using conditions (i)--(iii) in Lemma \ref{lem:type1cond} one finds
$(E_{r-1}A^*E_r)_{0,0} \neq 0$ and $(E_r A^* E_{r-1})_{0,0} \neq 0$.
The result follows.
\end{proof}

\begin{proofof}{Proposition \ref{prop:type1even}}
Consider the sequence $\Phi = (A, \{E_i\}_{i=0}^d, A^*, \{E^*_i\}_{i=0}^d)$.
We check conditions (i)--(v) in Definition \ref{def:LS}.
By Lemma \ref{lem:type1eventhths}, each of $A,A^*$ is multiplicity-free, so condition (i) holds.
By the construction conditions (ii) and (iii) hold.
By Lemma \ref{lem:type1evenEiAsEj} condition (iv) holds.
By Lemmas \ref{lem:irred} and  \ref{lem:type1evenEr-1AsEr} condition (v) holds.
Thus $\Phi$ is a Leonard system, and so $A,A^*$ is a Leonard pair.
One can show that $A,A^*$ has parameter array in Proposition \ref{prop:type1closed}
in a similar way as in the proof of Proposition \ref{prop:type2ex}.
\end{proofof}

\section{Proof of Proposition \ref{prop:(ii)->(i)}}
\label{sec:proofmain}

\begin{proofof}{Proposition \ref{prop:(ii)->(i)}}
Let $A,A^*$ be a Leonard pair on $V$ that is isomorphic to its opposite.
Let $\beta$ be the fundamental parameter of $A,A^*$, and let
\[
  (\{\th_i\}_{i=0}^d, \{\th^*_i\}_{i=0}^d, \{\vphi_i\}_{i=1}^d, \{\phi_i\}_{i=1}^d)
\]
be a parameter array of $A,A^*$.

(i):
By replacing $A,A^*$ with their nonzero scalar multiples if necessary,
we may assume that the parameter array is as in Proposition \ref{prop:type2closed}.
Let $B,B^*$ be the zero-diagonal TD-TD pair \eqref{eq:TDTD} in $\Mat_{d+1}(\F)$
with the values of $\{x_i\}_{i=1}^d$, $\{y_i\}_{i=1}^d$, $\{z_i\}_{i=1}^d$ as in
Proposition \ref{prop:type2ex}.
We show that $A,A^*$ is represented by $B.B^*$.
By Proposition \ref{prop:type2ex} the parameter array of $B,B^*$ is as in
Proposition \ref{prop:type2closed}.
So $A,A^*$ and $B,B^*$ have the same parameter array, and therefore
$A,A^*$ and $B,B^*$ are isomorphic by Lemma \ref{lem:unique}.
Thus $A,A^*$ is represented by $B,B^*$.

(ii), (iii):
Similar.
\end{proofof}

\section{Evaluating the Askey-Wilson relations}
\label{sec:evalAWrel}

For nonzero scalars $\{x_i\}_{i=1}^d$, $\{y_i\}_{i=1}^d$, $\{z_i\}_{i=1}^d$,
consider the zero-diagonal TD-TD pair \eqref{eq:TDTD} in $\Mat_{d+1}(\F)$; 
denote this pair by $A,A^*$.
Assume $A,A^*$ be a Leonard pair in $\Mat_{d+1}(\F)$ with fundamental 
parameter $\beta$. 
By Note \ref{note:(i)-(ii)} $A,A^*$ is isomorphic to its opposite.
Let
\[
  (\{\th_i\}_{i=0}^d, \{\th^*_i\}_{i=0}^d, \{\vphi_i\}_{i=1}^d, \{\phi_i\}_{i=1}^d)
\]
be a parameter array of $A,A^*$.
We consider the Askey-Wilson relations for $A,A^*$.
Let the scalars $\gamma$, $\gamma^*$, $\varrho$, $\varrho^*$ be from 
Lemma \ref{lem:gammarho}, and the scalars $\omega$, $\eta$, $\eta^*$ be
from Lemma \ref{lem:omega}.
By Lemmas \ref{lem:AWtype2}--\ref{lem:AWtype1} we have
$\gamma=0$, $\gamma^*=0$, $\eta=0$, $\eta^*=0$.
So the Askey-Wilson relations \eqref{eq:AW1orig}, \eqref{eq:AW2orig} become
\begin{align}
  A^2 A^* - \beta A A^* A + A^* A^2 - \varrho A^* - \omega A &= 0,   \label{eq:AW1}
\\
  {A^*}^2 A - \beta A^* A A^* + A {A^*}^2 - \varrho^* A - \omega A^* &= 0.  \label{eq:AW2}
\end{align}
By \eqref{eq:TDTD},  for $0 \leq i,j \leq d$
\begin{align*}
 A_{i,j} &=
  \begin{cases}
     1  &  \text{ if $i-j=1$},  \\
     z_j &  \text{ if $j-i=1$},  \\
     0  &  \text{ if $|i-j| \neq 1$},
  \end{cases}
&
 A^*_{i,j} &=
  \begin{cases}
     x_i  &  \text{ if $i-j=1$},  \\
     y_j z_j &  \text{ if $j-i=1$},  \\
     0  &  \text{ if $|i-j| \neq 1$}.
  \end{cases}
\end{align*}
For notational convenience, 
set $x_i=0$, $y_i=0$, $z_i=0$ for $i \leq 0$ or $i>d$.

\begin{lemma}                 \samepage
Assume $A,A^*$ satisfies \eqref{eq:AW1}.
Then for $2 \leq i \leq d-1$
\begin{equation}
  x_{i-1} - \beta x_i + x_{i+1} = 0,            \label{eq:1}
\end{equation}
\begin{equation}
 y_{i-1} - \beta y_i + y_{i+1} = 0.           \label{eq:2}
\end{equation}
\end{lemma}

\begin{proof}
Compute the $(i+1,i-2)$-entry of \eqref{eq:AW1} to get \eqref{eq:1}.
Compute the $(i-2,i+1)$-entry of  \eqref{eq:AW1} to get \eqref{eq:2}.
\end{proof}

\begin{lemma}                 \samepage
Assume $A,A^*$ satisfies \eqref{eq:AW2}.
Then for $2 \leq i \leq d-1$
\begin{equation}
  x_{i-1} x_i - \beta x_{i-1} x_{i+1} + x_i x_{i+1} = 0,   \label{eq:3}
\end{equation}
\begin{equation}
 y_{i-1} y_i - \beta y_{i-1} y_{i+1} + y_i y_{i+1} = 0.   \label{eq:4}
\end{equation}
\end{lemma}

\begin{proof}
Compute the $(i+1,i-2)$-entry of \eqref{eq:AW2} to get \eqref{eq:3}.
Compute the $(i-2,i+1)$-entry of \eqref{eq:AW2} to get \eqref{eq:4}
\end{proof}

\begin{lemma}                 \samepage
Assume $A,A^*$ satisfies \eqref{eq:AW1}.
Then for $1 \leq i \leq d$
\begin{align}
& z_{i-1} ( x_i - \beta x_{i-1} + y_{i-1} ) + z_i (2 x_i - \beta y_i) \notag
\\
&\qquad\qquad\qquad + z_{i+1} ( x_i - \beta x_{i+1} + y_{i+1}) - \varrho x_i - \omega =0,
                                                                                         \label{eq:5}
\\
& z_{i-1} ( y_i - \beta y_{i-1} + x_{i-1} ) + z_i  (2 y_i - \beta x_i) \notag
\\
&\qquad\qquad\qquad + z_{i+1}  ( y_i - \beta y_{i+1} + x_{i+1}) - \varrho y_i - \omega =0.
                                                                                         \label{eq:6}
\end{align}
\end{lemma}

\begin{proof}
Compute the $(i,i-1)$-entry of \eqref{eq:AW1} to get \eqref{eq:5}.
Compute the $(i-1,i)$-entry of \eqref{eq:AW1} to get \eqref{eq:6}.
\end{proof}

\begin{lemma}                 \samepage
Assume $A,A^*$ satisfies \eqref{eq:AW2}.
Then for $1 \leq i \leq d$
\begin{align}
& z_{i-1} ( y_{i-1} x_{i-1} - \beta y_{i-1} x_i + x_{i-1} x_i) 
        + z_i (2 y_i x_i - \beta x_i^2 ) \notag
\\
&\qquad\qquad\qquad 
         + z_{i+1}  (y_{i+1} x_{i+1} - \beta y_{i+1} x_i + x_i x_{i+1}) 
                 - \varrho^* - \omega x_i = 0,                           \label{eq:7}
\\
& z_{i-1} (x_{i-1} y_{i-1} - \beta x_{i-1} y_i + y_{i-1} y_i) 
        + z_i (2 x_i y_i - \beta y_i^2 ) \notag
\\
&\qquad\qquad\qquad 
         + z_{i+1}  (x_{i+1} y_{i+1} - \beta x_{i+1} y_i + y_i y_{i+1}) 
                 - \varrho^* - \omega y_i = 0.                           \label{eq:8}
\end{align}
\end{lemma}

\begin{proof}
Compute the $(i,i-1)$-entry of \eqref{eq:AW2} to get \eqref{eq:7}.
Compute the $(i-1,i)$-entry of \eqref{eq:AW2} to get \eqref{eq:8}.
\end{proof}

\section{Some equations}
\label{sec:equat}

For nonzero scalars $\{x_i\}_{i=1}^d$, $\{y_i\}_{i=1}^d$, $\{z_i\}_{i=1}^d$ in $\F$,
define $A,A^* \in \Mat_{d+1}(\F)$ by
\begin{align*} 
A &=
 \begin{pmatrix}
  0  & z_1 & & & & \text{\bf 0} \\
  1 & 0  & z_2 \\
   & 1 & 0 & \cdot \\
   & & \cdot & \cdot & \cdot \\
   & & & \cdot & \cdot & z_d \\
  \text{\bf 0}  & & & & 1 & 0
 \end{pmatrix},
& 
A^* &=
 \begin{pmatrix}
   0 & \yb_1 & & & & \text{\bf 0} \\
  x_1 & 0 & \yb_2  \\
   & x_2 & 0 & \cdot  \\
   & & \cdot & \cdot & \cdot \\
   & & & \cdot & \cdot & \yb_{d} \\
  \text{\bf 0} & & & & x_d & 0
 \end{pmatrix},
\end{align*}
where $\yb_i = y_i z_i$ for $0 \leq i \leq d$.
Assume $A,A^*$ is a Leonard pair with parameter array
\begin{equation}
  (\{\th_i\}_{i=0}^d, \{\th^*_i\}_{i=0}^d, \{\vphi_i\}_{i=1}^d, \{\phi_i\}_{i=1}^d). \label{eq:parrayE}
\end{equation}
By Lemma \ref{lem:split} there exists a basis for $\F^{d+1}$, with respect to which
the matrices representing $A,A^*$ are
\begin{align*} 
A &: \;\; 
 \begin{pmatrix}
  \th_0  & & & & & \text{\bf 0} \\
  1 & \th_1  \\
   & 1 & \th_2 \\
   & & \cdot & \cdot \\
   & & & \cdot & \cdot \\
  \text{\bf 0}  & & & & 1 & \th_{d}
 \end{pmatrix},
& 
A^*  &: \;\;
 \begin{pmatrix}
  \th^*_0 & \vphi_1 & & & & \text{\bf 0} \\
     & \th^*_1 & \vphi_2  \\
   &  & \th^*_2 & \cdot  \\
   & &    & \cdot & \cdot \\
   & & &    & \cdot & \vphi_d \\
  \text{\bf 0} & & & &  & \th^*_d
 \end{pmatrix},
\end{align*}
Denote the above matrices by $B$, $B^*$.
By the construction, there exists an invertible matrix $P \in \Mat_{d+1}(\F)$
such that both $AP=PB$ and $A^*P = P B^*$.
To simplify notation, define $P_{i,j}=0$ if $i$ or $j$ is not in $\{0,1,\ldots,d\}$. 

\begin{lemma}     \label{lem:entries}   \samepage
\ifDRAFT {\rm lem:entries}. \fi
For $0 \leq i,j \leq d$,
\begin{equation}
  P_{i-1,j} + z_{i+1} P_{i+1,j} - \th_j P_{i,j} - P_{i,j+1} = 0,    \label{eq:eq}
\end{equation}
\begin{equation}
 x_i P_{i-1,j} + y_{i+1}z_{i+1} P_{i+1,j} - \th^*_j P_{i,j}- \vphi_j P_{i,j-1} = 0.   \label{eq:eqs}
\end{equation}
\end{lemma}

\begin{proof}
Compute the $(i,j)$-entry of  $AP-PB$ and $A^*P-PB^*$.
\end{proof}

\begin{lemma}    \label{lem:P00nonzero}    \samepage
\ifDRAFT {\rm lem:P00nonzero}. \fi
We have $P_{0,0} \neq 0$ and $P_{d,d} \neq 0$.
\end{lemma}

\begin{proof}
By \eqref{eq:eqs} for $j=0$,
\begin{align*}
   y_{i+1}z_{i+1} P_{i+1,0} &= \th^*_0 P_{i,0} - x_i P_{i-1,0}  &&  (1 \leq i \leq d).
\end{align*}
Solving this recursion, we find that $P_{i,0}$ is a scalar multiple of $P_{0,0}$ for $0 \leq i \leq d$.
So, if $P_{00}=0$, then $0$th column of $P$ is zero; this contradicts that $P$ is invertible.
Therefore $P_{00} \neq 0$.
By \eqref{eq:eq} for $j=d$,
\begin{align*}
  P_{i-1,d}  &= \th_d P_{i,d}  - z_{i+1} P_{i+1,d}  &&  (0 \leq i \leq d).
\end{align*}
Solving this recursion, we find that $P_{i,d}$ is a scalar multiple of $P_{d,d}$ for $0 \leq i \leq d$.
So $P_{d,d} \neq 0$.
\end{proof}

\begin{lemma}    \label{lem:twoequations}    \samepage
\ifDRAFT {\rm lem:twoequations}. \fi
We have
\begin{equation} 
 y_1^2 z_1 (x_1-y_2) + y_1 y_2 \big( \vphi_1 + \th_0(\th^*_0 - \th^*_1) \big)
  - \th^*_0 (\th^*_0 y_1 - \th^*_1 y_2) = 0,                                   \label{eq:equat1}
\end{equation}
\begin{equation}
z_d(x_{d-1} - y_d) + \vphi_d + \th^*_d (\th_d - \th_{d-1})
    + \th_d (\th_{d-1} x_d - \th_d x_{d-1}) = 0.                                 \label{eq:equat2}
\end{equation}
\end{lemma}

\begin{proof}
We first show \eqref{eq:equat1}.
By \eqref{eq:eqs} for $(i,j)=(0,0)$, $(1,0)$, $(0,1)$ and \eqref{eq:eq} for $(i,j)=(0,0)$, $(1,0)$,
\[
 - \th^*_0 P_{0,0} + y_1 z_1 P_{1,0} = 0,
\]
\[
 x_1 P_{0,0} - \th^*_0 P_{1,0} + y_2 z_2 P_{2,0} = 0,
\]
\[
  - \vphi_1 P_{0,0} - \th^*_1 P_{0,1} + y_1 z_1 P_{1,1} = 0,
\]
\[
 - \th_0 P_{0,0} - P_{0,1} + z_1 P_{1,0} = 0,
\]
\[
 P_{0,0} - \th_0 P_{1,0} - P_{1,1} + z_2 P_{2,0} = 0.
\]
In these equation, eliminate $P_{1,0}$, $P_{0,1}$, $P_{1,1}$, $P_{2,0}$ to find that
$P_{0,0}$ times the left-hand side of \eqref{eq:equat1} is zero.
By this and Lemma \ref{lem:P00nonzero} we get \eqref{eq:equat1}.
Next we show \eqref{eq:equat2}.
By \eqref{eq:eq} for $(i,j)=(d,d)$, $(d-1,d)$, $(d,d-1)$ and \eqref{eq:eqs}
for $(i,j)=(d,d)$, $(d-1,d)$,
\[
 P_{d-1,d} - \th_d P_{d,d} = 0,
\]
\[
 P_{d-2,d} - \th_d P_{d-1,d} + z_d P_{d,d} = 0,
\]
\[
 P_{d-1,d-1} - \th_{d-1} P_{d,d-1} - P_{d,d} = 0,
\]
\[
 x_d P_{d-1,d} - \vphi_d P_{d,d-1} - \th^*_d P_{d,d} = 0,
\]
\[
 x_{d-1} P_{d-2,d} - \vphi_d P_{d-1,d-1} - \th^*_d P_{d-1,d} + y_d z_d P_{d,d} = 0.
\]
In these equations, eliminate $P_{d-1,d}$, $P_{d,d-1}$, $P_{d-1,d-1}$, $P_{d-2,d}$
to find that $P_{d,d}$ times the left-hand side of \eqref{eq:equat2} is zero.
By this and Lemma \ref{lem:P00nonzero} we get \eqref{eq:equat2}.
\end{proof}

\section{Proof of Theorem \ref{thm:main2}(i)}
\label{sec:type2}

\begin{proofof}{Theorem \ref{thm:main2}(i)}
Let $A,A^*$ be a zero-diagonal TD-TD Leonard pair in $\Mat_{d+1}(\F)$
with fundamental parameter $\beta=2$.
By replacing $A,A^*$ with their nonzero scalar multiples,
we may assume that $A,A^*$ has parameter array in Proposition \ref{prop:type2closed}.
Note by Lemma \ref{lem:type2cond} that
$\text{Char}(\F)$ is $0$ or greater than $d$, and $s^2 \neq 1$.
In view of Note \ref{note:subdiagonal1}, we may assume that the subdiagonal entries of $A$
are all $1$.
We show that $A,A^*$ is the pair \eqref{eq:TDTD} with $\{x_i\}_{i=1}^d$,
$\{y_i\}_{i=1}^d$, $\{z_i\}_{i=1}^d$ as in Proposition \ref{prop:type2ex}.
We use the Askey-Wilson relations for $A,A^*$.
By Lemma \ref{lem:AWtype2} $\varrho=4$, $\varrho^*=4$,
and $\omega = - 2(s+s^{-1})$.
Using \eqref{eq:1} and \eqref{eq:3} one finds
\begin{align}
  x_i &= x_1  && (1 \leq i \leq d).         \label{eq:type2xi}
\end{align}
Using \eqref{eq:2} and \eqref{eq:4} one finds
\begin{align}
  y_i &= y_1  && (1 \leq i \leq d).         \label{eq:type2yi}
\end{align}
By \eqref{eq:5} and \eqref{eq:6} for $i=1$ together with
\eqref{eq:type2xi} and \eqref{eq:type2yi},
\begin{equation}
  x_1 + y_1 = s + s^{-1}.                     \label{eq:type2x1y1}
\end{equation}
By \eqref{eq:5} for $i=1$ together with
\eqref{eq:type2xi}--\eqref{eq:type2x1y1},
\begin{equation}
   (s+s^{-1} - 2 x_1)(z_2 - 2 z_1 + 2) = 0.   \label{eq:type2x1z1z2}
\end{equation}
We claim that $s+s^{-1}-2 x_1 \neq 0$.
By way of contradiction, assume $s+s^{-1}-2 x_1 =0$,
so $x_1 = (s+s^{-1})/2$.
By this and \eqref{eq:type2x1y1}, $y_1 = (s+s^{-1})/2$.
Using these comments and \eqref{eq:type2xi}, \eqref{eq:type2yi}, we evaluate \eqref{eq:7}
to find that $(s-s^{-1})^2=0$, contradicting $s^2 \neq 1$.
Thus the claim holds.
By the claim and \eqref{eq:type2x1z1z2},
\begin{equation}
z_2 - 2 z_1 + 2 = 0.                             \label{eq:type2z1z2}
\end{equation}
In \eqref{eq:7} for $i=1$, eliminate $y_1$ using \eqref{eq:type2x1y1},
and eliminate $z_2$ using \eqref{eq:type2z1z2},
\begin{equation}  
  (x_1 -s)(x_1 - s^{-1}) = 0.                        \label{eq:type2x1}
\end{equation}
Thus either $x_1 = s$ or $x_1 = s^{-1}$.
By replacing $s$ with $s^{-1}$ if necessary, we may assume $x_1 = s$.
By this and \eqref{eq:type2xi} $x_i = s$ for $1 \leq i \leq d$.
By $x_1=s$ and \eqref{eq:type2x1y1} $y_1 = s^{-1}$.
By this and \eqref{eq:type2xi}, \eqref{eq:type2yi},
\begin{align}
  x_i &= s  &&  (1 \leq i \leq d),           \label{eq:type2xi2}
\\
  y_i &= s^{-1}  &&  (1 \leq i \leq d).    \label{eq:type2yi2}
\end{align}
By \eqref{eq:5} with \eqref{eq:type2xi2}, \eqref{eq:type2yi2},
\begin{align*}
  (s - s^{-1})(z_{i-1} - 2 z_i + z_{i+1} -2) = 0    &&  (2 \leq i \leq d-1).
\end{align*}
By this and $s^2 \neq 1$,
\begin{align*}
  z_{i-1} - 2 z_i + z_{i+1} -2 = 0    &&  (2 \leq i \leq d-1).
\end{align*}
Solve this recursion with \eqref{eq:type2z1z2} to find
\begin{align*}
  z_i &= i (z_1 - i+1)  && (1 \leq i \leq d).    
\end{align*}
So it suffices to show $z_1=d$.
Using \eqref{eq:type2xi}, \eqref{eq:type2yi}, we simplify \eqref{eq:equat1} to find
\[
   s^{-2} (s - s^{-1}) (z_1 -d) = 0.
\]
This forces $z_1 = d$.
The result follows.
\end{proofof}

\section{Proof of Theorem \ref{thm:main2}(ii)}
\label{sec:type3}

\begin{proofof}{Theorem \ref{thm:main2}(ii)}
Let $A,A^*$ be a zero-diagonal TD-TD Leonard pair in $\Mat_{d+1}(\F)$
with fundamental parameter $\beta=-2$.
By replacing $A,A^*$ with their nonzero scalar multiples,
we may assume that $A,A^*$ has parameter array in Proposition \ref{prop:type3closed}.
Note by Lemma \ref{lem:type3cond} that
$\text{Char}(\F)$ is $0$ or greater than $d$,
and 
$\tau$ is not among $1-d$, $3-d$, \ldots, $d-1$.
By the assumption of Theorem \ref{thm:main2} $d+1$ does not vanish in $\F$.
We show that $A,A^*$ is the pair \eqref{eq:TDTD} with $\{x_i\}_{i=1}^d$,
$\{y_i\}_{i=1}^d$, $\{z_i\}_{i=1}^d$ as in Proposition \ref{prop:type2ex}.
We use the Askey-Wilson relations for $A,A^*$.
By Lemma \ref{lem:AWtype3} $\varrho=4$, $\varrho^*=4$,
and $\omega = 4(d+1)\tau$.
By \eqref{eq:1} and \eqref{eq:3},
\begin{align}
  x_i &= (-1)^{i-1} x_1  && (1 \leq i \leq d).         \label{eq:type3xi}
\end{align}
By \eqref{eq:2} and \eqref{eq:4},
\begin{align}
  y_i &= (-1)^{i-1} y_1  && (1 \leq i \leq d).         \label{eq:type3yi}
\end{align}
By \eqref{eq:5}, \eqref{eq:6} for $i=1$ together with
\eqref{eq:type3xi}, \eqref{eq:type3yi},
\begin{equation}
  y_1 = x_1                 \label{eq:type3x1y1}
\end{equation}
By \eqref{eq:5} for $1 \leq i \leq d-1$, and using
\eqref{eq:type3xi}--\eqref{eq:type3x1y1},
\begin{align}
 z_i &=
  \begin{cases}
   i z_1 - i \big( (d+1)\tau x_1^{-1} + i - 1 \big) & \text{ if $i$ is even},
 \\
   i z_1 - (i-1) \big( (d+1)\tau x_1^{-1} + i \big) & \text{ if $i$ is odd}
  \end{cases} 
   &&  (1 \leq i \leq d).                                       \label{eq:type3zi}
\end{align}
By \eqref{eq:equat2}
\[
  2 x_1 (d + d \tau x_1 - x_1^2 z_1) = 0.
\]
So
\[
  z_1 = d(1 + \tau x_1) x_1^{-2}.
\]
By this and \eqref{eq:type3zi}
\begin{align}
 z_i &=
  \begin{cases}
    i (d x_1^{-1}-i+1) - i x_1^{-1} \tau  & \text{ if $i$ is even},
  \\
  i (d x_1^{-1}-i+1) + (d-i+1)x_1^{-1} \tau & \text{ if $i$ is odd}.
 \end{cases}
\end{align}
By these comments and \eqref{eq:equat2},
\[
   2 d (x_1 -1)(x_1+1)=0.
\]
So either $x_1 = 1$ or $x_1 = -1$.
Setting $\epsilon = x_1$ we find that
 $\{x_i\}_{i=1}^d$, $\{y_i\}_{i=1}^d$, $\{z_i\}_{i=1}^d$
are as in Proposition \ref{prop:type3ex}.
The result follows.
\end{proofof}

\section{Proof of Theorem \ref{thm:main2}(iii)}
\label{sec:type1}

In this section we prove Theorem \ref{thm:main2}(iii).
Let $A,A^*$ be a zero-diagonal TD-TD Leonard pair in $\Mat_{d+1}(\F)$.
Let $\beta$ be the fundamental parameter of $A,A^*$, and assume $\beta \neq 2$, $\beta \neq -2$.
By replacing $A,A^*$ with their nonzero scalar multiples,
we may assume that $A,A^*$ has parameter array in Proposition \ref{prop:type1closed}
for nonzero $q$, $s \in \F$.
We assume $q$ is not a root of unity.
Note by Lemma \ref{lem:char} that $\text{Char}(\F) \neq 2$.
By Lemma \ref{lem:type1cond}
\begin{align}
  s^2 &\neq q^i && (0 \leq i \leq d-1).                    \label{eq:conds}
\end{align}
We use the Askey-Wilson relations for $A,A^*$.
By Lemma \ref{lem:AWtype1}
\begin{align*}
\varrho &= q^{d-2}(q^2-1)^2,  \qquad\qquad  \varrho^* = q^{d-2}(q^2-1)^2,
\\
\omega &= - q^{-1} (q-1)^2 (q^{d+1}+1) \tau,
\end{align*}
where
\[
   \tau = s + s^{-1} q^{d-1}.
\]

\begin{lemma}    \label{lem:xiyi}    \samepage
\ifDRAFT {\rm lem:xiyi}. \fi
The following hold:
\begin{itemize}
\item[\rm (i)]
Either $\;x_i x_{i-1}^{-1} = q$ $\,(1 \leq i \leq d)\;$
or $\;x_i x_{i-1}^{-1} = q^{-1}$ $\,(1 \leq i \leq d)$.
\item[\rm (ii)]
Either $\;y_i y_{i-1}^{-1} = q$ $\,(1 \leq i \leq d)\;$
or $\;y_i y_{i-1}^{-1} = q^{-1}$ $\,(1 \leq i \leq d)$.
\end{itemize}
\end{lemma}

\begin{proof}
(i):
By \eqref{eq:1} for $i=2$,
\[
  x_1 - (q+q^{-1}) x_2 + x_3 = 0.
\]
By \eqref{eq:3} for $i=2$
\[
  x_1 x_2 - (q+q^{-1}) x_1 x_3 + x_2 x_3 = 0.
\]
In the above two equations, eliminate $x_3$ to find
\[
  (q+q^{-1})(x_1 - x_2 q)(x_1 - x_2 q^{-1})=0.
\]
We have $q+q^{-1} \neq 0$ since $q$ is not a root of unity.
Therefore either $x_2 x_1^{-1} = q$ or $x_2 x_1^{-1} = q^{-1}$.
Now solve the recursion \eqref{eq:1} to get the result.

(ii) Similar.
\end{proof}

\noindent
By Lemma \ref{lem:xiyi} we have four cases:
\begin{itemize}
\item[]
Case 1: $x_i x_{i-1}^{-1} = q^{-1}$ and $y_i y_{i-1}^{-1} = q^{-1}$ for $1 \leq i \leq d$.
\item[]
Case 2: $x_i x_{i-1}^{-1} = q^{-1}$ and $y_i y_{i-1}^{-1} = q$ for $1 \leq i \leq d$.
\item[]
Case 3: $x_i x_{i-1}^{-1} = q$ and $y_i y_{i-1}^{-1} = q^{-1}$ for $1 \leq i \leq d$.
\item[]
Case 4: $x_i x_{i-1}^{-1} = q$ and $y_i y_{i-1}^{-1} = q$ for $1 \leq i \leq d$.
\end{itemize}
Observe  that Case 3 is reduced to Case 2 by replacing $A,A^*$ with its anti-diagonal transpose.
Similarly Case 4 is reduced to Case 1.
So we consider Case 1 and Case 2.

\subsection{Case 1}

In this subsection we consider Case 1. So
\begin{align}                       
  x_i &= x_1 q^{1-i},  &  y_i &= y_1 q^{1-i}   && (1 \leq i \leq d).           \label{eq:case1xiyi}
\end{align}

\begin{lemma}   \samepage
We have
\begin{align}
 (q x_1 - y_1) z_1 &= (q-1)(q^d-1) \big( q^{d-1}(q+1)y_1^{-1} - \tau \big),   \label{eq:case1equat1}
\\
 (q x_1 - y_1) z_d &= q^{d-1}(q-1)(q^d-1) \big( (q+1)x_1 - \tau \big).         \label{eq:case1equat2}
\end{align}
\end{lemma}

\begin{proof}
Follows from \eqref{eq:equat1} and \eqref{eq:equat2}.
\end{proof}

\begin{lemma}               \samepage
For $1 \leq i \leq d$
\begin{align}
 & (x_1 - y_1 q) z_{i+1} - q^2 \big( 2 x_1 - (q+q^{-1}) y_1 \big) z_i + q^4 (x_1 - y_1 q^{-1}) z_{i-1}  \notag
\\
 & \qquad\qquad\qquad\qquad
   +  x_1 q^d (q^2-1)^2 - \tau q^{i} (q-1)^2 (q^{d+1}+1) = 0,   \label{eq:case1equ5}
\\
 & (y_1 - x_1 q) z_{i+1} - q^2 \big( 2 y_1 - (q+q^{-1}) x_1 \big) z_i + q^4 (y_1 - x_1 q^{-1}) z_{i-1} \notag
\\
 & \qquad\qquad\qquad\qquad
  +  y_1 q^d (q^2-1)^2 - \tau q^{i} (q-1)^2 (q^{d+1}+1) = 0,     \label{eq:case1equ6}
\\
 & (x_1 - y_1 q) z_{i+1} + q \big( 2 y_1 - (q+q^{-1}) x_1 \big) z_i + q^2 (x_1 - y_1 q^{-1}) z_{i-1}    \notag
\\
 &  \qquad\qquad\qquad\qquad
  -  x_1^{-1} q^{d+2i-3} (q^2-1)^2 + \tau q^{i-1} (q-1)^2 (q^{d+1}+1)=0.  \label{eq:case1equ7}
\end{align}
\end{lemma}

\begin{proof}
Follows from \eqref{eq:5}--\eqref{eq:7}.
\end{proof}

\begin{lemma}    \label{lem:case1equ56}  \samepage
\ifDRAFT {\rm lem:case1equ56}. \fi
Assume $x_1 \neq y_1$.
Then for $1 \leq i \leq d$
\begin{equation}
 (x_1 + y_1) \big( z_i - q^2 z_{i-1} -  q^d(q^2-1) \big) 
         + \tau q^{i-1} (q-1)(q^{d+1}+1)=0.                      \label{eq:case1equ56}
\end{equation}
\end{lemma}

\begin{proof}
In \eqref{eq:case1equ5} and \eqref{eq:case1equ6}, eliminate $z_{i+1}$ to find that
$x_1-y_1$ times \eqref{eq:case1equ56} is $0$.
\end{proof}

\medskip

First consider the case $x_1 \neq y_1$, $x_1 \neq - y_1$.

\begin{lemma}    \label{lem:206}    \samepage
\ifDRAFT {\rm lem:206}. \fi
Assume  $x_1 \neq y_1$ and $x_1 \neq - y_1$.
Then $\tau=x_1+y_1$, $x_1 y_1 = q^{d-1}$, and for $1 \leq i \leq d$
\begin{equation}
    z_i = q^{i-1} (q^i-1)(q^{d-i+1}-1).               \label{eq:case11zi}
\end{equation}
\end{lemma}

\begin{proof}
By  \eqref{eq:case1equ56}, for $1 \leq i \leq d$
\[
  z_i - q^2 z_{i-1} -  q^d (q^2-1)
      + \frac{\tau q^{i-1}(q-1)(q^{d+1}+1)}
                {x_1 + y_1} =0.  
\]
Solve this recursion with $z_0=0$ to find that for $1 \leq i \leq d$,
\begin{equation}
  z_i =  q^d (q^{2i}-1) 
       - \frac{\tau q^{i-1}(q^i-1)(q^{d+1}+1)}
                 {x_1 + y_1}.                                \label{eq:case11zi2}
\end{equation}
By this and \eqref{eq:case1equat2},
\[
  (q^d-1)(x_1+y_1- \tau) \big( (q^{d+2}+1) x_1 - (q^{d+1}+q)y_1 \big)=0.
\]
So either $\tau=x_1+y_1$ or $(q^{d+1}+1) x_1 = (q^{d+1}+q) y_1$.

First assume $\tau = x_1 + y_1$.
By \eqref{eq:case11zi2} with $\tau=x_1+y_1$ we get \eqref{eq:case11zi}.
By \eqref{eq:case1equat1},
\[
   (q^2-1)(q^d-1) (q^{d-1}-x_1y_1) y_1 = 0.
\]
So $x_1 y_1 = q^{d-1}$.

Next assume $\tau \neq x_1 + y_1$.
Then 
\[
   y_1 (q^{d+1}+q) = x_1 (q^{d+2}+1).
\]
We have $q^{d+2}+1 \neq 0$; otherwise both $q^{d+2}+1=0$ and $q^{d+1}+q=0$, so $q^2=1$,
contradicting Lemma \ref{lem:type1cond}.
Similarly, $q^d + 1 \neq 0$.
Also note that $q^{d+1}+1 \neq 0$; otherwise 
$y_1 (q+1)= x_1 (q+1)$ and so $x_1 = y_1$, contradicting the assumption.
Now we find
\[
  y_1 = \frac{(q^{d+2}+1) x_1}{q(q^d+1)}.   
\]
By \eqref{eq:case1equ7},
\[
  \tau = \frac{ q^{d-1}(x_1+y_1) (x_1 y_1 q^2+1)}
                  {x_1y_1 (q^{d+1}+1)}.  
\]
By \eqref{eq:case1equat1} with the above comments,
\[
    (q^{d+2}+1) x_1^2 = q^d (q^d+1).
\]
Using these comments we find $\tau = x_1 + y_1$, a contradiction.
\end{proof}

\medskip

Next consider the case $x_1 = - y_1$.

\begin{lemma}     \label{lem:207}   \samepage
\ifDRAFT {\rm lem:207}. \fi
Assume $x_1 = - y_1$.  Then $\tau=0$, $x_1^2 = - q^{d-1}$,
and for $1 \leq i \leq d$
\begin{equation}
 z_i = q^{i-1} (q^i-1)(q^{d-i+1}-1).                 \label{eq:case12zi}
\end{equation}
\end{lemma}

\begin{proof}
First assume $q^{d+1} \neq -1$.
Then $\tau=0$ by  \eqref{eq:case1equ56}.
In \eqref{eq:case1equ5} and \eqref{eq:case1equ7},
eliminate $z_{i+1}$ to find that for $1 \leq i \leq d$
\[
  z_i - q z_{i-1} - q^d(q-1)(1 +  x_1^{-2} q^{2i-3}) = 0. 
\]
Solving this recursion with $z_0=0$,
\begin{equation}
  z_i = q^d (q^i-1)(1 +  x_1^{-2} q^{i-2}).         \label{eq:case12zi2}
\end{equation}
By \eqref{eq:case1equat2},
\[
   (q+1)(q^d-1)(x_1^2 + q^{d-1})=0.
\]
So $x_1^2 = - q^{d-1}$. Now \eqref{eq:case12zi} follows from \eqref{eq:case12zi2}.

Next assume $q^{d+1}=-1$.
By \eqref{eq:case1equat1}
\[
  z_1 =
   \frac{ (q-1)(q^2 x_1 \tau - q - 1)}
          {q^3 x_1^2}.
\]
By this and \eqref{eq:case1equ5} for $i=1$,
\[
  z_2 = \frac{(q^2-1) \big( q(q-1) x_1^2 + q^2 \tau x_1 - q -1)}
                 {q^2 x_1^2}.
\]
Using these comments, evaluate \eqref{eq:case1equ7} for $i=1$ to find
\[
  (q^2-1)^2 (q^2 x_1^2 + q^2 \tau x_1 - 1)=0.
\]
By this
\[
    \tau = -x_1 + x_1^{-1} q^{-2}.
\]
Using these comments, solve the recursion \eqref{eq:case1equ5} to find that
\begin{align*}
  z_ i  &= - q^{-1} (q^i-1)(1 + x_1^{-2} q^{i-2})   &&  (1 \leq i \leq d).
\end{align*}
By this and \eqref{eq:case1equ7} for $i=d$,
\[
   2(q+1)(x_1^2-q^{-2}) = 0.
\]
So $x_1^2 =  q^{-2}$.
By these comments, $\tau=0$ and
\begin{align*}
  z_i &= - q^{-1} (q^i-1)(1+q^i)  &&  (1 \leq i \leq d).
\end{align*}
Now the result follows.
\end{proof}

\medskip

Next consider the case $x_1 = y_1$.

\begin{lemma}             \label{lem:7.18}        \samepage
\ifDRAFT {\rm lem:7.18}. \fi
Assume $x_1 = y_1$. Then $d$ is even, $\tau = x_1 + x_1^{-1} q^{d-1}$,
and for $1 \leq i \leq d$
\begin{equation}                                                \label{eq:case13zi}
 z_i =
  \begin{cases}
     q^d (q^i-1) (1 - x_1^{-2} q^{i-2})          &   \text{ if $i$ is even},
  \\
   - q^{i-1} (q^{d-i+1}-1)(1 -  x_1^{-2} q^{d+i-1})   &   \text{ if $i$ is odd}.
  \end{cases}
\end{equation}             
\end{lemma}

\begin{proof}
In \eqref{eq:case1equ5} and \eqref{eq:case1equ7}, eliminate $z_{i+1}$ to find
\[
 z_i + q z_{i-1} + q^d(q+1)(1 +  x_1^{-2} q^{2i-3})
                         - \tau x_1^{-1} q^{i-1} (q^{d+1}+1)=0.  
\]
Solving this recursion with $z_0 = 0$, we get
\begin{equation}                                                    \label{eq:case13ziaux}
 z_i =
  \begin{cases}
       q^d (q^i-1)(1 -  x_1^{-2} q^{i-2})               &   \text{ if $i$ is even},
  \\
  - q^d (q^i+1) (1 +  x_1^{-2} q^{i-2})  + \tau x_1^{-1} q^{i-1}(q^{d+1}+1)   & \text{ if $i$ is odd}.
  \end{cases}  
\end{equation}
By this and \eqref{eq:case1equat1},
\[
   q^{d-1}(q^2-1)(x_1 + x_1^{-1} q^{d-1} - \tau) = 0.
\]
So $\tau =  x_1 + x_1^{-1} q^{d-1}$.
By this and \eqref{eq:case13ziaux} we get \eqref{eq:case13zi}.
We show $d$ is even.
By way of contradiction, assume $d$ is odd.
By \eqref{eq:case1equat2}
\[
  z_d = q^d (q^d-1) (1- x_1^{-2} q^{d-2}).
\]
By \eqref{eq:case13zi}
\[
 z_d = - q^{d-1} (q-1) (1-x_1^{-2} q^{2d-1}).
\]
Comparing these two equations,
\[
  (q^{d+1} -1) (x_1^2 - q^{d-1})=0.
\]
So either $q^{d+1}=1$ or $x_1^2 = q^{d-1}$.
First assume $q^{d+1}=1$.
Set $r = (d+1)/2$, and observe $r$ is an integer such that $2 \leq i \leq d-1$.
We have $q^r = \pm 1$, contradicting Lemma \ref{lem:type1cond}(i), (ii).
Next assume $x_1^2 = q^{d-1}$.
By $\tau = x_ 1 + x_1^{-1} q^{d-1}$ and $\tau = s + s^{-1} q^{d-1}$
we have either $x_1 = s$ or $x_1 = s^{-1} q^{d-1}$.
In either case $s^2 = q^{d-1}$, contradicting Lemma \ref{lem:type1cond}(iii).
\end{proof}

\begin{lemma}    \label{lem:AWsolutions1}       \samepage
\ifDRAFT {\rm lem:AWsolutions1}. \fi
At least one of the following holds:
\begin{itemize}
\item[\rm (i)]
$x_1 y_1 = q^{d-1}$, $\tau = x_1 + x_1^{-1} q^{d-1}$,
and for $1 \leq i \leq d$
\begin{align*}
z_i &= q^{i-1} (q^{i}-1)(q^{d-i+1}-1).
\end{align*}
\item[\rm (ii)]
$x_1 = y_1$, $\tau = x_1 + x_1^{-1} q^{d-1}$,
and for $1 \leq i \leq d$
\begin{align*}
z_i &=
  \begin{cases}
    q^d(q^i-1)(1-x_1^{-2} q^{i-2} )     & \text{ if $i$ is even},
  \\
  - q^{i-1} (q^{d-i+1}-1)(1- x_1^{-2} q^{d+i-1})  & \text{ if $i$ is odd}.
  \end{cases}
\end{align*}
\end{itemize}
\end{lemma}

\begin{proof}
First assume $x_1 \neq y_1$ and $x_1 \neq - y_1$.
Then case (i) occurs by Lemma \ref{lem:206}.
Next assume $x_1 = -y_1$.
Then case (i) occurs by Lemma \ref{lem:207}.
Next assume $x_1 = y_1$.
Then case (ii) occurs by Lemma \ref{lem:7.18}.
\end{proof}

\subsection{Case 2}

In this subsection we consider Case 2. So
\begin{align}                       
  x_i &= x_1 q^{1-i},  &  y_i &= y_1 q^{i-1}   && (1 \leq i \leq d).           \label{eq:case2xiyi}
\end{align}

\begin{lemma}
We have
\begin{align}
 (x_1 - q y_1) z_1 &= (q-1)(q^d-1) \big( (q+1) y_1^{-1} - q \tau \big),   \label{eq:case2equat1}
\\
 (x_1- q^{2d-3} y_1) z_d &= q^{d-2} (q-1)(q^d-1) \big( (q+1) x_1 - \tau \big).  \label{eq:case2equat2}
\end{align}
\end{lemma}

\begin{proof}
Follows from \eqref{eq:equat1} and \eqref{eq:equat2}.
\end{proof}

\begin{lemma}
For $1 \leq i \leq d$
\begin{align}   
  & (x_1 - y_1 q^{2i+1})z_{i+1} -  q^2 \big( 2 x_1  - y_1 q^{2i-2} (q+q^{-1}) \big) z_i
       + q^4 (x_1 - y_1 q^{2i-5}) z_{i-1}                                             \notag
\\
  &  \qquad\qquad\qquad \qquad\qquad
      + x_1 q^d (q^2-1)^2 - \tau q^{i} (q-1)^2 (q^{d+1}+1) = 0,  \label{eq:equ5}
\\
 & (x_1 - y_1 q^{2i+1}) z_{i+1} -  q \big( x_1 (q+q^{-1}) - 2 y_1 q^{2i-2} \big) z_i
    + q^2 (x_1 - y_1 q^{2i-5}) z_{i-1}                                                       \notag
\\
 & \qquad\qquad\qquad\qquad
  -  y_1 q^{d+2i-3} (q^2-1)^2 + \tau  q^{i-1} (q-1)^2 (q^{d+1}+1) = 0,  \label{eq:equ6}
\\
 & (x_1 - y_1 q^{2i+1}) z_{i+1} -  q \big( x_1 (q+q^{-1}) - 2 y_1 q^{2i-2} \big) z_i
    + q^2 (x_1 - y_1 q^{2i-5}) z_{i-1}                                                       \notag
\\
 & \qquad\qquad\qquad\qquad
  -  x_1^{-1} q^{d+2i-3} (q^2-1)^2 + \tau q^{i-1} (q-1)^2 (q^{d+1}+1) = 0.  \label{eq:equ7}
\end{align}
\end{lemma}

\begin{proof}
Follows from \eqref{eq:5}--\eqref{eq:7}.
\end{proof}

\begin{lemma}   \label{lem:x1y1}   \samepage
\ifDRAFT {\rm lem:x1y1}. \fi
We have $x_1 y_1 = 1$.
\end{lemma}

\begin{proof}
Compare \eqref{eq:equ6} and \eqref{eq:equ7}.
\end{proof}

\begin{lemma}
For $1 \leq i \leq d$
\begin{align}
 & (x_1^2 - q^{2i-1}) z_i - q^2 (x_1^2 - q^{2i-5}) z_{i-1}   \notag
\\
 &  \qquad\qquad
    -  q^d (x_1^2 + q^{2i-3}) (q^2-1) + \tau x_1 q^{i-1} (q-1)(q^{d+1}+1) =0.   \label{eq:equ56}
\end{align}
\end{lemma}

\begin{proof}
In  \eqref{eq:equ5} and \eqref{eq:equ6}, eliminate $z_{i+1}$ and simplify the result using
$y_1 = x_1^{-1}$.
\end{proof}

\begin{lemma}
For $1 \leq i \leq d$
\begin{align}
&  (x_1^2 - q^{2i-3})(x_1^2 - q^{2i-1}) z_i        \notag
\\
& \qquad\qquad
  =  (q^i-1)(x_1^2 - q^{i-2}) 
  \big( q^d (q^i+1)(x_1^2 + q^{i-2}) - \tau x_1 q^{i-1} (q^{d+1}+1) \big).       \label{eq:equat}
\end{align}
\end{lemma}

\begin{proof}
Solve the recursion  \eqref{eq:equ56} with $z_0=0$.
\end{proof}

\begin{lemma}                \label{lem:cases}
\ifDRAFT {\rm lem:cases}. \fi
We have
\begin{equation}                         \label{eq:case2tau}
  \tau = x_1 + x_1^{-1} q^{d-1}.
\end{equation}
\end{lemma}

\begin{proof}
In \eqref{eq:case2equat1} and \eqref{eq:equ56} for $i=1$, eliminate $z_1$ to find
\[
    (q^2-1)^2 (x_1 + x_1^{-1} q^{d-1} - \tau)=0.
\]
So \eqref{eq:case2tau} holds.
\end{proof}

\begin{lemma}       \label{lem:8.12}    \samepage
\ifDRAFT {\rm lem:8.12}. \fi
For $1 \leq i \leq d$
\begin{align}
 & (x_1^2 - q^{2i-1}) z_i - q^2 (x_1^2 - q^{2i-5}) z_{i-1}           \notag
\\
 & \qquad\qquad
   - q^d (q^2-1)(x_1^2 + q^{2i-3}) +  q^{i-1}(q-1)(q^{d+1}+1)(x_1^2 + q^{d-1}) =0.  \label{eq:equ56b}
\end{align}
\end{lemma}

\begin{proof}
Follows from \eqref{eq:equ56} and \eqref{eq:case2tau}.
\end{proof}

\begin{lemma}
For $1 \leq i \leq d$
\begin{equation}                                                \label{eq:(ii)}
 (x_1^2 - q^{2i-3})(x_1^2 - q^{2i-1}) z_i
 =  q^{i-1}(q^i-1)(q^{d-i+1}-1)(x_1^2 - q^{i-2})(x_1^2 - q^{d+i-1}).
\end{equation}
\end{lemma}

\begin{proof}
Follows from \eqref{eq:equat} and \eqref{eq:case2tau}.
\end{proof}

\begin{lemma}     \label{lem:8.15}    \samepage
\ifDRAFT {\rm lem:8.15}. \fi
We have $x_1^2 \neq q^{2i-1}$ for $1 \leq i \leq d-1$, and
\begin{align}   
 z_1 &= \frac{(q-1)(q^d-1)(x_1^2 - q^d)}
                  {x_1^2 - q},                                       \label{eq:(ii)z1}
\\
 z_i &= \frac{ q^{i-1}(q^i-1)(q^{d-i+1}-1)(x_1^2 - q^{i-2})(x_1^2 - q^{d+i-1})}
               {(x_1^2 - q^{2i-3})(x_1^2 - q^{2i-1})}          && (2 \leq i \leq d-1).  \label{eq:(ii)zi}
\end{align}   
\end{lemma}

\begin{proof}
By \eqref{eq:case2equat1}
\[
   (x_1^2 - q) z_1 = (q-1)(q^d-1)(x_1^2 - q^d).
\]
We have $x_1^2 \neq q$; otherwise $0 = (q-1)(q^d-1)(q-q^d)$, contradicting
Lemma \ref{lem:type1cond}(i).
So \eqref{eq:(ii)z1} holds.
We claim that $x_1^2 \neq q^{2i-1}$ for $2 \leq i \leq d-1$.
By way of contradiction, assume $x_1^2 = q^{2i-1}$ for some $i$ $(2 \leq i \leq d-1)$.
In \eqref{eq:(ii)}, the left-hand side vanishes, so
\[
  0=  q^{i-1}(q^i-1)(q^{d-i+1}-1) (q^{2i-1}- q^{i-2})(q^{2i-1}-q^{d+i-1}).
\]
This is a contradiction by Lemma \ref{lem:type1cond}(i).
So the claim holds.
Now \eqref{eq:(ii)zi} follows from \eqref{eq:(ii)}.
\end{proof}

\begin{lemma}     \label{lem:8.16}    \samepage
\ifDRAFT {\rm lem:8.16}. \fi
Assume $x_1^2 = q^{2d-1}$.
Then
\begin{align}
 z_i  &= \frac{ q^{d-i}(q^i-1)(q^{2d-i+1}-1)}
               { (q^{d-i} + 1)(q^{d-i+1} + 1)}    && (1 \leq i \leq d-1),  \label{eq:case2zi3}
\\
 z_d &= \frac{(q^d-1)(q^{d+1}-1)}
                  {q+1}.                                                               \label{eq:case2zd3}
\end{align}
\end{lemma}

\begin{proof}
Line \eqref{eq:case2zi3} follows from Lemma \ref{lem:8.15}.
Line \eqref{eq:case2zd3} follows from \eqref{eq:case2equat2}.
\end{proof}

\begin{lemma}     \label{lem:8.17}    \samepage
\ifDRAFT {\rm lem:8.17}. \fi
Assume  $x_1^2 \neq q^{2d-1}$.
Then
\begin{equation}
  z_d = \frac{  q^{d-1} (q-1) (q^d-1) (x_1^2 - q^{d-2})} {x_1^2 - q^{2d-3}}.
\end{equation}
\end{lemma}

\begin{proof}
Follows from \eqref{eq:case2equat2}.
\end{proof}

\begin{lemma}    \label{lem:AWsolutions2}     \samepage
\ifDRAFT {\rm lem:AWsolutions2}. \fi
We have $x_1 y_1 = 1$, and at least one of the following holds:
\begin{itemize}
\item[\rm (i)]
$x_1^2 =  q^{2d-1}$, $\tau = x_1 + x_1^{-1} q^{d-1}$, and
\begin{align*}
  z_i &= \frac{q^{d-i}(q^i-1)(q^{2d-i+1}-1)}
                {(q^{d-i}+1)(q^{d-i+1}+1)}                && (1 \leq i \leq d-1),
\\
 z_d &= \frac{(q^d-1)(q^{d+1}-1)}
                  {q+1}.
\end{align*}
\item[\rm (ii)]
$x_1^2 \neq q^{2i-1}$ for $1 \leq i \leq d$,
$\tau = x_1 + x_1^{-1} q^{d-1}$,
and
\begin{align*}
 z_1 &= \frac{(q-1)(q^d-1)(x_1^2 - q^d)}
                  {x_1^2 - q},
\\
 z_i &= \frac{q^{i-1}(q^i-1)(q^{d-i+1}-1)(x_1^2 - q^{i-2})(x_1^2 - q^{d+i-1})}
                 {(x_1^2 - q^{2i-3})(x_1^2 - q^{2i-1})}                  && (2 \leq i \leq d-1),
\\
 z_d &= \frac{q^{d-1}(q-1)(q^d-1)(x_1^2 - q^{d-2})}
                  {x_1^2 - q^{2d-3}}.
\end{align*}
\end{itemize}
\end{lemma}

\begin{proof}
First assume $x_1^2 = q^{2d-1}$.
Then case (i) occurs by Lemma \ref{lem:8.16}.
Next assume $x_1^2 \neq q^{2d-1}$ 
Then case (ii) occurs by Lemmas \ref{lem:8.15} and \ref{lem:8.17}.
\end{proof}

\subsection{Completing the proof of Theorem \ref{thm:main2}(iii)}

By Lemmas \ref{lem:AWsolutions1} and
\ref{lem:AWsolutions2},
we have one of cases (i), (ii) in Lemma \ref{lem:AWsolutions1}
and cases (i), (ii) in Lemma \ref{lem:AWsolutions2}.
In either case we have $\tau = x_1 + x_1^{-1} q^{d-1}$.
By this and $\tau = s + s^{-1} q^{d-1}$, we have either $x_1=s$ or $x_1 = s^{-1}q^{d-1}$.
In view of Note \ref{note:s},
we may assume $x_1 = s$ by replacing $s$ with $s^{-1} q^{d-1}$ if necessary.
First assume (i) in Lemma \ref{lem:AWsolutions1} occurs.
Then $\{x_i\}_{i=1}^d$, $\{y_i\}_{i=1}^d$, $\{z_i\}_{i=1}^d$
are as in Proposition \ref{prop:type1compact}.
Next assume case (ii) in Lemma \ref{lem:AWsolutions1} occurs.
Then $\{x_i\}_{i=1}^d$, $\{y_i\}_{i=1}^d$, $\{z_i\}_{i=1}^d$
are as in Proposition \ref{prop:type1even}.
Next assume case (i) in Lemma \ref{lem:AWsolutions2} occurs.
Then $\{x_i\}_{i=1}^d$, $\{y_i\}_{i=1}^d$, $\{z_i\}_{i=1}^d$
are as in Proposition \ref{prop:type1LT} with $s^2 = q^{2d-1}$.
Next assume case (ii) in Lemma \ref{lem:AWsolutions2} occurs.
We have $s^2 \neq q^{2i-1}$ for $1 \leq i \leq d$.
By Lemma \ref{lem:type1cond}(iii) $s^2 \neq q^{2i}$ for $0 \leq i \leq d-1$.
So $s^2 \neq q^i$ for $0 \leq i \leq 2d-1$.
Now $\{x_i\}_{i=1}^d$, $\{y_i\}_{i=1}^d$, $\{z_i\}_{i=1}^d$
are as in Proposition \ref{prop:type1LT}.
This completes the proof of Theorem \ref{thm:main2}(iii).

\section{Acknowledgement}

The author thanks Paul Terwilliger for many insightful comments that lead to
great improvements in the paper.

%

\bigskip\bigskip\noindent
Kazumasa Nomura\\
Tokyo Medical and Dental University\\
Kohnodai, Ichikawa, 272-0827 Japan\\
email: knomura@pop11.odn.ne.jp

\medskip\noindent
{\small
{\bf Keywords.} Leonard pair, tridiagonal pair, Askey-Wilson relation, orthogonal polynomial
\\
\noindent
{\bf 2010 Mathematics Subject Classification.} 05E35, 05E30, 33C45, 33D45
}

\end{document}